\documentclass[reqno]{amsart}
\usepackage{amssymb}
\usepackage{amsfonts}
\usepackage{amsthm}
\usepackage{amsmath}
\usepackage{hyperref}
\newcommand{\C}{{\mathbb{C}}}
\newcommand{\R}{{\mathbb{R}}}

\let\Im=\undefined\DeclareMathOperator*{\Im}{Im}

\newcommand{\rad}{\text{rad}}

\newtheorem{theorem}{Theorem}[section]

\newtheorem{lemma}[theorem]{Lemma}

\newtheorem{corollary}[theorem]{Corollary}
\newtheorem{conjecture}[theorem]{Conjecture}
\newtheorem{proposition}[theorem]{Proposition}
\theoremstyle{definition}
\newtheorem{definition}[theorem]{Definition}
\theoremstyle{remark}
\newtheorem*{remark}{Remark}
\newtheorem*{remarks}{Remarks}

\newcounter{smalllist}
\newenvironment{SL}{\begin{list}{{\rm(\roman{smalllist})\hss}}{%
\setlength{\topsep}{0mm}\setlength{\parsep}{0mm}\setlength{\itemsep}{0mm}%
\setlength{\labelwidth}{2.0em}\setlength{\itemindent}{2.5em}\setlength{\leftmargin}{0em}\usecounter{smalllist}%
}}{\end{list}}

\newenvironment{CI}{\begin{list}{{\ $\bullet$\ }}{%
\setlength{\topsep}{0mm}\setlength{\parsep}{0mm}\setlength{\itemsep}{0mm}%
\setlength{\labelwidth}{0mm}\setlength{\itemindent}{-1.5em}\setlength{\leftmargin}{1.5em}%
\setlength{\labelsep}{0mm} }}{\end{list}}

\newcommand{\eps}{{\varepsilon}}

\begin{document}

\title[The mass-critical NLS with radial data]
{The mass-critical nonlinear Schr\"odinger equation with radial data in dimensions three and higher}
\author{Rowan Killip}
\address{University of California, Los Angeles}
\author{Monica Visan}
\address{Institute for Advanced Study}
\author{Xiaoyi Zhang}
\address{Academy of Mathematics and System Sciences, Chinese Academy of Sciences}
% \keywords{nonlinear Schr\"odinger equation, %well-posedness}

%\vspace{-0.3in}
\begin{abstract}
We establish global well-posedness and scattering for solutions to the mass-critical nonlinear Schr\"odinger equation $iu_t +
\Delta u = \pm |u|^{4/d} u$ for large spherically symmetric $L^2_x(\R^d)$ initial data in dimensions $d\geq 3$.
In the focusing case we require that the mass is strictly less than that of the ground state.
As a consequence, we obtain that in the focusing case, any spherically symmetric blowup solution must concentrate
at least the mass of the ground state at the blowup time.
\end{abstract}

\maketitle

%\tableofcontents

\section{Introduction}

%\subsection{The mass-critical nonlinear Schr\"odinger equation}

The $d$-dimensional mass-critical nonlinear Schr\"odinger equation is given by
\begin{equation}\label{nls}
i u_t + \Delta u = F(u) \quad \text{with } F(u) := \mu |u|^{\frac4d} u
\end{equation}
where $u$ is a complex-valued function of spacetime $\R\times\R^d$.  Here $\mu=\pm1$, with $\mu=1$ known as the
defocusing equation and $\mu=-1$ as the focusing equation.

The name `mass-critical' refers to the fact that the scaling symmetry
\begin{equation}\label{scaling}
u(t,x) \mapsto  u_\lambda(t,x):= \lambda^{-\frac d2} u( \lambda^{-2} t, \lambda^{-1} x)
\end{equation}
leaves both the equation and the mass invariant.  The mass of a solution is
\begin{equation}\label{massdef}
M(u(t)) := \int_{\R^d} |u(t,x)|^2\,dx
\end{equation}
and is conserved under the flow.

In this paper, we investigate the Cauchy problem for \eqref{nls} for spherically symmetric $L^2_x(\R^d)$ initial data in dimensions
$d\geq 3$ by adapting the recent argument from \cite{KTV}, which treated the case $d=2$.  Before describing our results, we need to
review some background material.  We begin by making the notion of a solution more precise:

\begin{definition}[Solution] A function $u: I \times \R^d \to \C$ on a non-empty time interval $I \subset \R$
is a \emph{solution} (more precisely, a strong $L^2_x(\R^d)$ solution) to \eqref{nls} if
it lies in the class $C^0_t L^2_x(K \times \R^d) \cap L^{2(d+2)/d}_{t,x}(K \times \R^d)$ for all compact $K \subset I$, and
obeys the Duhamel formula
\begin{align}\label{old duhamel}
u(t_1) = e^{i(t_1-t_0)\Delta} u(t_0) - i \int_{t_0}^{t_1} e^{i(t_1-t)\Delta} F(u(t))\ dt
\end{align}
for all $t_0, t_1 \in I$.  Note that by Lemma~\ref{L:strichartz} below, the condition $u\in L^{2(d+2)/d}_{t,x}$ locally in time
guarantees that the integral converges, at least in a weak-$L_x^2$ sense.
\end{definition}

\begin{remark}
The condition that $u$ is in $L^{2(d+2)/d}_{t,x}$ locally in time is natural.  This space appears in the
Strichartz inequality (Lemma~\ref{L:strichartz}); consequently, all solutions to the linear problem lie in this space.  Existence of
solutions to \eqref{nls} in this space is guaranteed by the local theory discussed below; it is also necessary in order to
ensure uniqueness of solutions in this local theory. Solutions to \eqref{nls} in this class have been intensively studied, see
for example \cite{begout, bourg.2d, ck, cwI, caz, keraani, mv, tao-lens, compact, tvz-higher, tsutsumi}.
\end{remark}

Associated to this notion of solution is a corresponding notion of blowup.  As we will see in Theorem~\ref{local} below,
this precisely corresponds to the impossibility of continuing the solution.

\begin{definition}[Blowup]\label{D:blowup}
We say that a solution $u$ to \eqref{nls} \emph{blows up forward in time} if there exists a time $t_0 \in I$ such that
$$ \int_{t_0}^{\sup I} \int_{\R^d} |u(t,x)|^{2(d+2)/d}\, dx \,dt = \infty$$
and that $u$ \emph{blows up backward in time} if there exists a time $t_0 \in I$ such that
$$ \int_{\inf I}^{ t_0} \int_{\R^d} |u(t,x)|^{2(d+2)/d}\, dx\, dt = \infty.$$
\end{definition}

The local theory for \eqref{nls} was worked out by Cazenave and Weissler \cite{cwI}.  They constructed local-in-time solutions
for arbitrary initial data in $L^2_x(\R^d)$; however, due to the critical nature of the equation, the resulting time of
existence depends on the profile of the initial data and not merely on its $L_x^2$-norm. Cazenave and Weissler also constructed
global solutions for small initial data.  We summarize their results in the theorem below.

\begin{theorem}[Local well-posedness, \cite{cwI, caz}]\label{local}
Given $u_0 \in L^2_x(\R^d)$ and $t_0 \in \R$, there exists a unique maximal-lifespan solution $u$ to \eqref{nls} with
$u(t_0)=u_0$. We will write $I$ for the maximal lifespan.  This solution also has the following properties:
\begin{CI}
\item (Local existence) $I$ is an open neighbourhood of $t_0$.
\item (Mass conservation) The solution $u$ obeys mass conservation: $M(u(t))=M(u_0)$ for all $t\in I$.
\item (Blowup criterion) If $\sup(I)$ or $\inf(I)$ is finite, then $u$ blows up in the corresponding time direction.
\item (Continuous dependence) The map that takes initial data to the corresponding strong solution is uniformly continuous
on compact time intervals for bounded sets of initial data.
\item (Scattering) If $\sup(I)=+\infty$ and $u$ does not blow up forward in time, then $u$ scatters forward in time, that is,
there exists a unique $u_+ \in L^2_x(\R^d)$ such that $$ \lim_{t \to +\infty} \| u(t)-e^{it\Delta} u_+ \|_{L^2_x(\R^d)} = 0.$$
Similarly, if $\inf(I)=-\infty$ and $u$ does not blow up backward in time, then $u$ scatters backward in time, that is, there is
a unique $u_- \in L^2_x(\R^d)$ so that
$$ \lim_{t \to -\infty} \| u(t)-e^{it\Delta} u_- \|_{L^2_x(\R^d)} = 0.$$
\item (Small data global existence) If $M(u_0)$ is sufficiently small depending on $d$, then $u$ is a global solution
with finite $L^{2(d+2)/d}_{t,x}$ norm.
\end{CI}
\end{theorem}

It is widely believed that in the defocusing case, all $L^2_x$ initial data lead to a global solution with finite $L^{2(d+2)/d}_{t,x}$
spacetime norm (and hence also scattering).

In the focusing case, the general consensus is more subtle.  Let $Q$ denote the \emph{ground state},
that is, the unique positive radial solution to
\begin{equation}\label{ground}
\Delta Q + Q^{1+4/d} = Q.
\end{equation}
(The existence and uniqueness of $Q$ was established in \cite{blions} and \cite{kwong} respectively.)
Then $u(t,x) := e^{it} Q(x)$ is a solution to \eqref{nls}, which is global but blows up both forward and backward in time
(in the sense of Definition~\ref{D:blowup}).  More dramatically, by applying the pseudoconformal transformation to $u$, we obtain a solution
\begin{equation}\label{pc}
v(t,x) := |t|^{-d/2} e^{i\frac{|x|^2-4}{4t}} Q\Bigl( \frac{x}{t} \Bigr)
\end{equation}
with the same mass that blows up in finite time.  It is widely believed that this ground state example is the minimal-mass obstruction
to global well-posedness and scattering in the focusing case.

To summarize, we subscribe to

\begin{conjecture}[Global existence and scattering]\label{conj}
Let $d \geq 1$ and $\mu=\pm 1$.  In the defocusing case $\mu=+1$, all maximal-lifespan solutions to \eqref{nls}
are global and do not blow up either forward or backward in time.  In the focusing case $\mu=-1$, all maximal-lifespan
solutions $u$ to \eqref{nls} with $M(u) < M(Q)$ are global and do not blow up either forward or backward in time.
\end{conjecture}

\begin{remark}
While this conjecture is phrased for $L^2_x(\R^d)$ solutions, it is equivalent to a scattering claim for smooth solutions;
see \cite{begout, carles, keraani, tao-lens}.  In \cite{blue,tao-lens}, it is also shown
that the global existence and the scattering claims are equivalent in the $L^2_x(\R^d)$ category.
\end{remark}

The contribution of this paper toward settling this conjecture is

\begin{theorem}\label{main}
Let $d\geq 3$.  Then Conjecture~\ref{conj} is true in the class of spherically symmetric initial data (for either choice of sign $\mu$).
\end{theorem}

Conjecture~\ref{conj} has been the focus of much intensive study and several partial results for various choices of $d, \mu$,
and sometimes with the additional assumption of spherical symmetry. The most compelling evidence in favour of this conjecture
stems from results obtained under the assumption that $u_0$ has additional regularity.  For the defocusing equation, it is easy
to prove global well-posedness for initial data in $H^1_x$; this follows from the usual contraction mapping argument combined with the
conservation of mass and energy; see, for example, \cite{caz}. Recall that the energy is given by
\begin{align}\label{energy}
E(u(t)):=\int_{\R^d}\frac 12 |\nabla u(t,x)|^2 +\mu \frac{d}{2(d+2)}|u(t,x)|^{\frac{2(d+2)}d} \, dx.
\end{align}
Note that for general $L_x^2$ initial data, the energy need not be finite.

The focusing equation with data in $H^1_x$ was treated by Weinstein \cite{weinstein}.  A key ingredient was his proof of the
sharp Gagliardo--Nirenberg inequality:

\begin{theorem}[Sharp Gagliardo--Nirenberg, \cite{weinstein}]\label{sGN}
\begin{equation}
\int_{\R^d} \bigl|f(x)\bigr|^{\frac{2(d+2)}{d}} \, dx   \leq
 \frac{d+2}{d} \biggl( \frac{\|f\|_{L^2}^2}{\|Q\|_{L^2}^2} \biggr)^{\frac{2}{d}} \int_{\R^d} \bigl|\nabla f(x)\bigr|^{2} \, dx.
\end{equation}
\end{theorem}

As noticed by Weinstein, this inequality implies that the energy \eqref{energy} is positive once $M(u_0)<M(Q)$; indeed, it gives
an upper bound on the $\dot H_x^1$-norm of the solution at all times of existence. Combining this with a contraction mapping
argument and the conservation of mass and energy, Weinstein proved global well-posedness for the focusing equation with initial
data in $H^1_x$ and mass smaller than that of the ground state.

Note that the iterative procedure used to obtain a global solution both for the defocusing and the focusing equations with
initial data in $H^1_x$ does not yield finite spacetime norms; in particular, scattering does not follow even for more regular
initial data.

In dimensions one and two, there has been much work \cite{bourg.2d, CGTz, CKSTT2, resonant, CRSW, PSST-1D, gf, tzirakis}
devoted to lowering the regularity of the initial data from $H_x^1$ toward $L^2_x(\R^d)$ and thus, toward establishing the conjecture.
For analogous results in higher dimensions, see \cite{PSST-hD, VZ-focusing}.

In the case of spherically symmetric solutions, Conjecture~\ref{conj} was recently settled in the high-dimensional defocusing case
$\mu=+1$, $d \geq 3$ in \cite{tvz-higher}; thus only the $\mu=-1$ case of Theorem~\ref{main} is new.  However, the techniques used
in \cite{tvz-higher} do not seem to be applicable to the focusing problem, primarily because the Morawetz inequality is no longer coercive in that case.
Instead, our argument is based on the recent preprint \cite{KTV}, which resolved the conjecture for $\mu=\pm1$, $d=2$, and spherically
symmetric data.  In turn, \cite{KTV} uses techniques developed to treat the analogous conjecture for the energy-critical problem,
such as \cite{borg:scatter, ckstt:gwp, RV, tao:gwp radial, thesis:art, Monica:thesis} and particularly~\cite{merlekenig}.  We will
give a more thorough discussion of the relation of the current work to these predecessors later, when we outline the argument.

\subsection{Mass concentration in the focusing problem}
Neither Theorem~\ref{main} nor Conjecture~\ref{conj} address the focusing problem for masses greater than or equal to that of
the ground state.  In this case, blowup solutions exist and attention has been focused on describing their properties.
For instance, finite-time blowup solutions with finite energy and mass equal to that of the ground state have been completely characterized by
Merle \cite{merle2}; they are precisely the ground state solution up to symmetries of the equation.

Several works have shown that finite-time blowup solutions must concentrate a positive amount of mass around the blowup time
$T^*$.  For finite energy data, see \cite{mt,nawa,weinstein-2} where it is shown that there exists $x(t)\in \R^d$ so that
$$
\liminf_{t\nearrow T^*}\int_{|x-x(t)|\leq R} \bigl|u(t,x)\bigr|^2 \,dx \geq M(Q)
$$
for any $R>0$.  For merely $L^2_x(\R^2)$ initial data, Bourgain \cite{bourg.2d} proved that some small amount of mass must concentrate
in parabolic windows (at least along a subsequence):
$$
\limsup_{t\nearrow T^*} \sup_{x_0\in\R^2} \int_{|x-x_0|\leq (T^*-t)^{1/2}} \bigl|u(t,x)\bigr|^2 \,dx \geq c,
$$
where $c$ is a small constant depending on the mass of $u$. This result was extended to other dimensions in \cite{begout,keraani}.

Combining Theorem~\ref{main} with the argument in \cite[\S10]{KTV}, one obtains the following concentration result.

\begin{corollary}[Blowup solutions concentrate the mass of the ground state]\label{cor:conc}\leavevmode\\
Let $d\geq 3$ and $\mu=-1$.  Let $u$ be a spherically symmetric solution to \eqref{nls} that blows up at time $0<T^*\leq \infty$.
If $T^*<\infty$, then there exists a sequence $t_n \nearrow T^*$ so that for any sequence $R_n\in (0,\infty)$ obeying $(T^*-t_n)^{-1/2} R_n\to \infty$,
\begin{align}\label{conc finite}
\limsup_{n\to \infty}\int_{|x|\leq R_n}|u(t_n,x)|^2\, dx\geq M(Q).
\end{align}
If $T^*=\infty$, then there exists a sequence $t_n \to\infty$ such that for any sequence $R_n\in (0,\infty)$ with $t_n^{-1/2}
R_n\to \infty$,
\begin{align}\label{conc infinite}
\limsup_{n\to \infty}\int_{|x|\leq R_n}|u(t_n,x)|^2\, dx\geq M(Q).
\end{align}
The analogous statement holds in the negative time direction.
\end{corollary}

\subsection{Outline of the proof}
Beginning with Bourgain's seminal work \cite{borg:scatter} on the energy-critical NLS, it has become apparent that in order to prove
spacetime bounds for general solutions, it is sufficient to treat a special class of solutions, namely, those that
are simultaneously localized in both frequency and space.  For further developments, see \cite{ckstt:gwp, RV, tao:gwp radial, thesis:art, Monica:thesis}.

A new and much more efficient alternative to Bourgain's induction on mass (or energy) method has recently been developed.  It uses
a (concentration) compactness technique to isolate minimal-mass/energy blowup solutions as opposed to the almost-blowup
solutions of the induction method.  Building on earlier developments in \cite{begout, bourg.2d, keraani-2, keraani, mv},
Kenig and Merle \cite{merlekenig} used this method to treat the energy-critical focusing problem with radial data in dimensions
three, four, and five.

To explain what the concentration compactness argument gives in our context, we need to introduce the following important notion:

\begin{definition}[Almost periodicity modulo scaling]\label{apdef}
Given $d \geq 1$ and $\mu = \pm 1$, a solution $u$ with lifespan $I$ is said to be \emph{almost periodic modulo scaling} if there
exists a (possibly discontinuous) function $N: I \to \R^+$ and a function $C: \R^+ \to \R^+$ such that
$$ \int_{|x| \geq C(\eta)/N(t)} |u(t,x)|^2\ dx \leq \eta$$
and
$$ \int_{|\xi| \geq C(\eta) N(t)} |\hat u(t,\xi)|^2\ d\xi \leq \eta$$
for all $t \in I$ and $\eta > 0$.  We refer to the function $N$ as the \emph{frequency scale function}
and to $C$ as the \emph{compactness modulus function}.
\end{definition}

\begin{remarks}
1. The parameter $N(t)$ measures the frequency scale of the solution at time $t$, and $1/N(t)$ measures the spatial scale; see
\cite{compact, tvz-higher} for further discussion.  Note that we have the freedom to modify $N(t)$ by any bounded function of
$t$, provided that we also modify the compactness modulus function $C$ accordingly. In particular, one could restrict $N(t)$ to
be a power of $2$ if one wished, although we will not do so here. Alternatively, the fact that the solution trajectory $t
\mapsto u(t)$ is continuous in $L^2_x(\R^d)$ can be used to show that the function $N$ may be chosen to depend
continuously on $t$.

2. By the Ascoli--Arzela Theorem, a family of functions is precompact in $L_x^2(\R^d)$ if and only if it is norm-bounded and
there exists a compactness modulus function $C$ so that
$$
\int_{|x| \geq C(\eta)} |f(x)|^2\ dx + \int_{|\xi| \geq C(\eta)} |\hat f(\xi)|^2\ d\xi \leq \eta
$$
for all functions $f$ in the family.  Thus, an equivalent formulation of Definition~\ref{apdef} is as follows: $u$ is almost
periodic modulo scaling if and only if
$$
\{ u(t): t \in I \} \subseteq \{ f(x/\lambda) : \lambda\in(0,\infty)\text{ and }f \in K \}.
$$
for some compact subset $K$ of $L^2_x(\R^d)$.
\end{remarks}

In \cite[Theorems~1.13 and~7.2]{compact} the following result was established (see also \cite{begout, keraani}), showing that any failure of
Conjecture~\ref{conj} must be `caused' by a very special type of solution.  For simplicity we state it only in the spherically symmetric case.

\begin{theorem}[Reduction to almost periodic solutions]\label{main-compact}
Fix $\mu$ and $d\geq2$.  Suppose that Conjecture~\ref{conj} fails for spherically symmetric data.  Then, there exists a
spherically symmetric maximal-lifespan solution $u$ which is almost periodic modulo scaling and which blows up both forward and backward in time,
and in the focusing case we also have $M(u) < M(Q)$.
\end{theorem}

In \cite{KTV}, this result was further refined so as to identify three specific enemies.  Once again, we state it only
in the spherically symmetric case.

\begin{theorem}[Three special scenarios for blowup, \cite{KTV}]\label{comp}
Fix $\mu$ and $d\geq 2$ and suppose that Conjecture~\ref{conj} fails for spherically symmetric data. Then there exists a
spherically symmetric maximal-lifespan solution $u$ which is almost periodic modulo scaling, blows up both forward and backward in time,
and in the focusing case also obeys $M(u) < M(Q)$.  Moreover, the solution $u$ may be chosen to match one of the following three scenarios:
\begin{CI}
\item (Soliton-like solution) We have $I = \R$ and
\begin{equation}\label{nsim}
 N(t) = 1
\end{equation}
for all $t \in \R$ (thus the solution stays in a bounded space/frequency range for all time).
\item (Double high-to-low frequency cascade) We have $I = \R$,
\begin{equation}\label{cascade}
\liminf_{t \to -\infty} N(t) = \liminf_{t \to +\infty} N(t) = 0,
\end{equation}
and
\begin{equation}\label{freqbound}
\sup_{t \in \R} N(t) < \infty
\end{equation}
for all $t \in I$.
\item (Self-similar solution) We have $I = (0,+\infty)$ and
\begin{equation}\label{ssb}
 N(t) = t^{-1/2}
\end{equation}
for all $t \in I$.
\end{CI}
\end{theorem}

In light of this result, the proof of Theorem~\ref{main} is reduced to showing that none of these three scenarios can occur.
In doing this, we follow the model set forth in \cite{KTV}.  In all cases, the key step is to prove that $u$ has additional regularity.
Indeed, to treat the first two scenarios, we need more than one derivative in $L_x^2$; for the self-similar scenario, $H^1_x$ suffices.
The possibility of showing such additional regularity stems from the fact
that $u$ is both frequency and space localized; this in turn is an expression of the fact that $u$ has minimal mass among all
blowup solutions.

A further manifestation of this minimality is the absence of a scattered wave at the endpoints of the lifespan $I$; more formally, we have

\begin{lemma}[{\cite[Section 6]{compact}}]\label{duhamel L}
Let $u$ be a solution to \eqref{nls} which is almost periodic modulo scaling on its maximal-lifespan $I$.  Then, for all $t\in I$,
\begin{equation}\label{duhamel}
\begin{aligned}
u(t)&=\lim_{T\nearrow\,\sup I}i\int_t^T e^{i(t-t')\Delta} F(u(t'))\,dt'\\
&=-\lim_{T\searrow\,\inf I}i\int_T^t e^{i(t-t')\Delta} F(u(t'))\,dt',
\end{aligned}
\end{equation}
as weak limits in $L_x^2$.
\end{lemma}

Another important property of solutions that are almost periodic modulo scaling is that the behaviour of the spacetime norm
is governed by that of $N(t)$.  More precisely, we have the following lemma from \cite{KTV}:

\begin{lemma}[Spacetime bound, \cite{KTV}]\label{spacelemma}
Let $u$ be a non-zero solution to \eqref{nls} with lifespan $I$, which is almost periodic modulo scaling
with frequency scale function $N: I \to \R^+$.  If $J$ is any subinterval of $I$, then
\begin{equation*}%\label{uj-conv}
\int_J N(t)^2\,dt \lesssim_u \int_J \int_{\R^d} |u(t,x)|^{\frac{2(d+2)}{d}}\, dx\, dt\lesssim_u 1 + \int_J N(t)^2\,dt.
\end{equation*}
\end{lemma}

The nonexistence of self-similar solutions is proved in Section~\ref{ss-sec}.  We first prove that any such solution would belong
to $C^0_t H_x^1$ and then observe that $H_x^1$ solutions are global (see the discussion after Theorem~\ref{main}),
while self-similar solutions are not.

For the remaining two cases, higher regularity is proved in Section~\ref{glob-sob-sec}.  In order to best take advantage of
Lemma~\ref{duhamel L}, we exploit a decomposition of spherically symmetric functions into incoming and outgoing waves; this is
discussed in Section~\ref{S:in/out}.

In Section~\ref{hilo cascade-sec}, we use the additional regularity together with the conservation of energy to preclude the
double high-to-low frequency cascade.  In Section~\ref{soliton-sec}, we disprove the existence of soliton-like solutions using a
truncated virial identity in much the same manner as \cite{merlekenig}.

As noted earlier, the argument just described is closely modelled on \cite{KTV}, which treated the same equation in two dimensions.
The main obstacle in extending that work to higher dimensions is the fractional power appearing in the nonlinearity.  This problem
presents itself when we prove additional regularity, which is already the most demanding part of \cite{KTV}.  Additional regularity
is proved via a bootstrap argument using Duhamel's formula.  However, fractional powers can downgrade regularity (a fractional
power of a smooth function need not be smooth); in particular, they preclude the simple Littlewood-Paley arithmetic that
is usually used in the case of polynomial nonlinearities.

The remedy is twofold: first we use fractional chain rules (see Lemmas~\ref{F Lip} and~\ref{fdfp}) that allow us to take more than one
derivative of a nonlinearity that is merely $C^{1+\frac4d}$ in $u$.  Secondly, we push through the resulting complexities in the
bootstrap argument.  An important role is played by Lemma~\ref{Gronwall} (a Gronwall-type result), which we use to untangle the intricate
relationship between frequencies in $u$ and those in $|u|^\frac4d u$.

\subsection*{Acknowledgements}
We are grateful to Shuanglin Shao for access to preliminary drafts of his work.

R.~K. was supported by NSF grants DMS-0701085 and DMS-0401277 and by
a Sloan Foundation Fellowship.  He is also grateful to the Institute
for Advanced Study (Princeton) for its hospitality. X.~Z. was
supported by NSF-10601060 and Project 973 in China. M.~V. was
supported under NSF grant DMS-0111298.

Any opinions, findings and conclusions or recommendations expressed are those of the authors and do not reflect the views of the
National Science Foundation.

\section{Notation and linear estimates}

This section contains the basic linear estimates we use repeatedly in the paper.

\subsection{Some notation}
We use $X \lesssim Y$ or $Y \gtrsim X$ whenever $X \leq CY$ for some constant $C>0$.  We use $O(Y)$ to denote any quantity $X$
such that $|X| \lesssim Y$.  We use the notation $X \sim Y$ whenever $X \lesssim Y \lesssim X$.  The fact that these constants
depend upon the dimension $d$ will be suppressed.  If $C$ depends upon some additional parameters, we will indicate this with
subscripts; for example, $X \lesssim_u Y$ denotes the assertion that $X \leq C_u Y$ for some $C_u$ depending on $u$.

We use the `Japanese bracket' convention $\langle x \rangle := (1 +|x|^2)^{1/2}$.

We write $L^q_t L^r_{x}$ to denote the Banach space with norm
$$ \| u \|_{L^q_t L^r_x(\R \times \R^d)} := \Bigl(\int_\R \Bigl(\int_{\R^d} |u(t,x)|^r\ dx\Bigr)^{q/r}\ dt\Bigr)^{1/q},$$
with the usual modifications when $q$ or $r$ are equal to infinity, or when the domain $\R \times \R^d$ is replaced by a smaller
region of spacetime such as $I \times \R^d$.  When $q=r$ we abbreviate $L^q_t L^q_x$ as $L^q_{t,x}$.

The next lemma is a variant of Gronwall's inequality that we will use to handle some bootstrap arguments below.
The proof given is a standard application of techniques from the theories of Volterra and Toeplitz operators.

\begin{lemma}[A Gronwall inequality]\label{Gronwall}
Fix $r\in(0,1)$ and $K\geq 4$.  Let $b_k$ be a bounded sequence of non-negative numbers and $x_k$
a sequence obeying $0\leq x_k\leq b_k$ for $0\leq k<K$ and
\begin{align}\label{Gron rec}
0\leq x_k \leq b_k + \sum_{l=0}^{k-K} r^{k-l} x_l \qquad \text{for all $k\geq K$.}
\end{align}
Then
\begin{align}\label{Gron bound}
0\leq x_k \lesssim \sum_{l=0}^{k} r^{k-l} \exp\bigl\{\tfrac{\log(K-1)}{K-1}\, (k-l)\bigr\} b_l
\end{align}
for all $k\geq 0$.  In particular, if $b_k=O(2^{-k\sigma})$ and $2^{\sigma} r (K-1)^{1/(K-1)} < 1$,
then $x_k=O(2^{-k\sigma})$.
\end{lemma}

\begin{proof}
Elementary arguments show that we need only obtain the bound for the case of equality, namely, where
\begin{equation}\label{Gron eq}
(1 - A) x = b.
\end{equation}
Here $x$ and $b$ denote the semi-infinite vectors built from the corresponding sequences, while $A$ is the matrix
with entries
$$
A_{k,l}= \begin{cases} r^{k-l} &:\text{if $k-l\geq K$,}\\ 0 &:\text{otherwise.}\end{cases}
$$

The triangular structure of $A$ guarantees that \eqref{Gron eq} can be solved (though not \emph{a priori} in $\ell^\infty$);
more precisely, it guarantees that the geometric series for $(1-A)^{-1}$ converges entry-wise.  To obtain bounds for the
entries of this inverse matrix, it is simplest to use a functional model: under the mapping of sequences to functions
$$
x_k \mapsto \sum_{k=0}^\infty x_k z^k \quad\text{and}\quad b_k \mapsto \sum_{k=0}^\infty b_k z^k,
$$
the matrix $A$ becomes multiplication by $r^K z^K(1-rz)^{-1}$.  In the same way, the entries of $(1-A)^{-1}$ come from
the Taylor coefficients of
$$
a(z):= \frac{1-rz}{1-rz-r^Kz^K}.
$$

Using $e^x\geq 1 +x$ with $x=-\log|rz|$, we see that
\begin{align*}
|1-rz| \geq \bigl(\tfrac{1}{r|z|} -1\bigr) r|z| \geq \frac{\log(K-1)}{K-1} r|z| \geq \log(K-1)\, r^K |z|^{K}
\end{align*}
on the disk $|z| \leq r^{-1}(K-1)^{-1/(K-1)}$.
This shows that $a(z)$ is bounded and analytic on this disk. (Note that the hypothesis $K\geq4$ implies that
$\log(K-1)>1$.)  The inequality \eqref{Gron bound} now follows from the standard Cauchy estimates.
\end{proof}

\subsection{Basic harmonic analysis}\label{ss:basic}
Let $\varphi(\xi)$ be a radial bump function supported in the ball $\{ \xi \in \R^d: |\xi| \leq \tfrac {11}{10} \}$ and equal to
$1$ on the ball $\{ \xi \in \R^d: |\xi| \leq 1 \}$.  For each number $N > 0$, we define the Fourier multipliers
\begin{align*}
\widehat{P_{\leq N} f}(\xi) &:= \varphi(\xi/N) \hat f(\xi)\\
\widehat{P_{> N} f}(\xi) &:= (1 - \varphi(\xi/N)) \hat f(\xi)\\
\widehat{P_N f}(\xi) &:= \psi(\xi/N)\hat f(\xi) := (\varphi(\xi/N) - \varphi(2\xi/N)) \hat f(\xi)
\end{align*}
and similarly $P_{<N}$ and $P_{\geq N}$.  We also define
$$ P_{M < \cdot \leq N} := P_{\leq N} - P_{\leq M} = \sum_{M < N' \leq N} P_{N'}$$
whenever $M < N$.  We will usually use these multipliers when $M$ and $N$ are \emph{dyadic numbers} (that is, of the form $2^n$
for some integer $n$); in particular, all summations over $N$ or $M$ are understood to be over dyadic numbers.  Nevertheless, it
will occasionally be convenient to allow $M$ and $N$ to not be a power of $2$. Note that $P_N$ is not truly a projection; to get
around this, we will occasionally need to use fattened Littlewood-Paley operators:
\begin{equation}\label{PMtilde}
\tilde P_N := P_{N/2} + P_N +P_{2N}.
\end{equation}
These obey $P_N \tilde P_N = \tilde P_N P_N= P_N$.

As with all Fourier multipliers, the Littlewood-Paley operators commute with the propagator $e^{it\Delta}$, as well as with
differential operators such as $i\partial_t + \Delta$. We will use basic properties of these operators many many times,
including

\begin{lemma}[Bernstein estimates]\label{Bernstein}
 For $1 \leq p \leq q \leq \infty$,
\begin{align*}
\bigl\| |\nabla|^{\pm s} P_N f\bigr\|_{L^p_x(\R^d)} &\sim N^{\pm s} \| P_N f \|_{L^p_x(\R^d)},\\
\|P_{\leq N} f\|_{L^q_x(\R^d)} &\lesssim N^{\frac{d}{p}-\frac{d}{q}} \|P_{\leq N} f\|_{L^p_x(\R^d)},\\
\|P_N f\|_{L^q_x(\R^d)} &\lesssim N^{\frac{d}{p}-\frac{d}{q}} \| P_N f\|_{L^p_x(\R^d)}.
\end{align*}
\end{lemma}

The next few results provide important tools for dealing with the fractional power appearing
in the nonlinearity.

\begin{lemma}[Fractional chain rule for a $C^1$ function, \cite{chris:weinstein}]\label{F Lip}
Suppose $G\in C^1(\mathbb C)$, $s \in (0,1]$, and $1<p,p_1,p_2<\infty$ such that $\frac 1p=\frac 1{p_1}+\frac 1{p_2}$.   Then,
$$
\||\nabla|^sG(u)\|_p\lesssim \|G'(u)\|_{p_1}\||\nabla|^s u\|_{p_2}.
$$
\end{lemma}

When the function $G$ is no longer $C^1$, but merely H\"older continuous, we have the following useful chain rule:

\begin{lemma}[Fractional chain rule for a H\"older continuous function, \cite{thesis:art}]\label{fdfp}
\ Let $G$ be a H\"older continuous function of order $0<\alpha<1$.  Then, for every $0<s<\alpha$, $1<p<\infty$,
and $\tfrac{s}{\alpha}<\sigma<1$ we have
\begin{align}\label{fdfp2}
\bigl\| |\nabla|^s G(u)\bigr\|_p
\lesssim \bigl\||u|^{\alpha-\frac{s}{\sigma}}\bigr\|_{p_1} \bigl\||\nabla|^\sigma u\bigr\|^{\frac{s}{\sigma}}_{\frac{s}{\sigma}p_2},
\end{align}
provided $\tfrac{1}{p}=\tfrac{1}{p_1} +\tfrac{1}{p_2}$ and $(1-\frac s{\alpha \sigma})p_1>1$.
\end{lemma}

\begin{corollary}\label{cor:s deriv}
Let $0\leq s<1+\frac 4d$.  Then, on any spacetime slab $I\times\R^d$ we have
\begin{align*}
\bigl\||\nabla|^s F(u)\bigr\|_{L_{t,x}^{\frac{2(d+2)}{d+4}}}
&\lesssim \bigl\||\nabla|^s u\bigr\|_{L_{t,x}^{\frac{2(d+2)}d}} \|u\|_{L_{t,x}^{\frac{2(d+2)}d}}^{\frac 4d}\\
\intertext{and}
\bigl\||\nabla|^s F(u)\bigr\|_{L_t^\infty L_x^{\frac{2r}{r+4}}}
&\lesssim \bigl\||\nabla|^s u\bigr\|_{L_t^\infty L_x^2} \|u\|_{L_t^\infty L_x^{\frac{2r}d}}^{\frac 4d},
\end{align*}
for any $\max\{d,4\}\leq r\leq \infty$.
\end{corollary}

\begin{proof}

Fix a compact interval $I$.  Throughout the proof, all spacetime estimates will be on $I\times\R^d$.

We begin with the first claim.  For $0<s\leq 1$, this is an easy consequence of Lemma~\ref{F Lip}.
We now address the case $1<s<1+\frac 4d$.  By the chain rule and the fractional product rule, we estimate
\begin{align*}
\bigl\||\nabla|^s F(u)\bigr\|_{L_{t,x}^{\frac{2(d+2)}{d+4}}}
&\lesssim \bigl\||\nabla|^{s-1} \bigl( \nabla u F_z(u) + \nabla \bar u F_{\bar z}(u) \bigr) \bigr\|_{L_{t,x}^{\frac{2(d+2)}{d+4}}}\\
&\lesssim \bigl\||\nabla|^s u\bigr\|_{L_{t,x}^{\frac{2(d+2)}d}} \|u\|_{L_{t,x}^{\frac{2(d+2)}d}}^{\frac 4d}\\
&\quad + \|\nabla u\|_{L_{t,x}^{\frac{2(d+2)}d}}
        \Bigl[\bigl\||\nabla|^{s-1} F_z(u)\bigr\|_{L_{t,x}^{\frac{d+2}2}} + \bigl\||\nabla|^{s-1} F_{\bar z}(u)\bigr\|_{L_{t,x}^{\frac{d+2}2}}\Bigr].
\end{align*}
The claim will follow from this, once we establish
\begin{align}\label{fractional}
\bigl\||\nabla|^{s-1} F_z(u)\bigr\|_{L_{t,x}^{\frac{d+2}2}} + \bigl\||\nabla|^{s-1} F_{\bar z}(u)\bigr\|_{L_{t,x}^{\frac{d+2}2}}
\lesssim \bigl\||\nabla|^\sigma u\bigr\|_{L_{t,x}^{\frac{2(d+2)}d}}^{\frac{s-1}\sigma} \|u\|_{L_{t,x}^{\frac{2(d+2)}d}}^{\frac 4d-\frac{s-1}\sigma}
\end{align}
for some $\frac{d(s-1)}4 <\sigma<1$.  Indeed, one simply has to note that by interpolation,
\begin{align*}
\bigl\||\nabla|^\sigma u\bigr\|_{L_{t,x}^{\frac{2(d+2)}d}}
&\lesssim \bigl\||\nabla|^s u\bigr\|_{L_{t,x}^{\frac{2(d+2)}d}}^{\frac \sigma s} \|u\|_{L_{t,x}^{\frac{2(d+2)}d}}^{1-\frac \sigma s}\\
\intertext{and}
\|\nabla u\|_{L_{t,x}^{\frac{2(d+2)}d}}
&\lesssim \bigl\||\nabla|^s u\bigr\|_{L_{t,x}^{\frac{2(d+2)}d}}^{\frac 1s} \|u\|_{L_{t,x}^{\frac{2(d+2)}d}}^{1-\frac 1s}.
\end{align*}

To derive \eqref{fractional}, we remark
that $F_z$ and $F_{\bar z}$ are H\"older continuous functions of
order $\tfrac 4d$ and use Lemma~\ref{cor:s deriv} (with
$\alpha:=\tfrac 4d$ and $s:=s-1$).

We now turn to the second claim.  Note that the condition $r\geq 4$ simply insures that $\tfrac{2r}{r+4}\geq 1$.
For $0<s\leq 1$, the claim follows immediately from Lemma~\ref{F Lip}.
Let us consider the case $1<s<1+\frac 4d$.  By the chain rule and the fractional product rule,
\begin{align*}
\bigl\||\nabla|^s F(u)\bigr\|_{L_t^\infty L_x^{\frac{2r}{r+4}}}
&\lesssim \bigl\||\nabla|^{s-1} \bigl( \nabla u F_z(u) + \nabla \bar u F_{\bar z}(u) \bigr) \bigr\|_{L_t^\infty L_x^{\frac{2r}{r+4}}}\\
&\lesssim \bigl\||\nabla|^s u\bigr\|_{L_t^\infty L_x^2} \|u\|_{L_t^\infty L_x^{\frac{2r}d}}^{\frac 4d}\\
&\quad + \|\nabla u\|_{L_t^\infty L_x^{\frac{2rs}{r+(s-1)d}}}
      \bigl\||\nabla|^{s-1} O\bigr(|u|^{\frac 4d}\bigl)\bigr\|_{L_t^\infty L_x^{\frac{2rs}{(r-d)(s-1)+4s}}}.
\end{align*}
By interpolation,
$$
\|\nabla u\|_{L_t^\infty L_x^{\frac{2rs}{r+(s-1)d}}}
\lesssim \bigl\||\nabla|^s u\bigr\|_{L_t^\infty L_x^2}^{\frac 1s} \|u\|_{L_t^\infty L_x^{\frac{2r}d}}^{1-\frac 1s}.
$$
Thus, the claim will follow once we establish
\begin{equation}\label{fcr goal}
\bigl\||\nabla|^{s-1} O\bigr(|u|^{\frac 4d}\bigl)\bigr\|_{L_t^\infty L_x^{\frac{2rs}{(r-d)(s-1)+4s}}}
\lesssim \bigl\||\nabla|^s u\bigr\|_{L_t^\infty L_x^2}^{1-\frac 1s} \|u\|_{L_t^\infty L_x^{\frac{2r}d}}^{\frac 4d +\frac 1s -1}.
\end{equation}
Applying Lemma~\ref{fdfp}, we obtain
\begin{align*}
\bigl\||\nabla|^{s-1} O\bigr(|u|^{\frac 4d}\bigl)\bigr\|_{L_t^\infty L_x^{\frac{2rs}{(r-d)(s-1)+4s}}}
&\lesssim \bigl\||\nabla|^\sigma  u\bigr\|_{L_t^\infty L_x^{\frac{2rs}{sd+\sigma(r-d)}}}^{\frac{s-1}\sigma}
    \|u\|_{L_t^\infty L_x^{\frac{2r}d}}^{\frac 4d-\frac{s-1}\sigma}
\end{align*}
for any $\tfrac{d(s-1)}4<\sigma<1$.  The inequality \eqref{fcr goal} now follows from
$$
\bigl\||\nabla|^\sigma  u\bigr\|_{L_t^\infty L_x^{\frac{2rs}{sd+\sigma(r-d)}}}
\lesssim \bigl\||\nabla|^s u\bigr\|_{L_t^\infty L_x^2}^{\frac{\sigma}s} \|u\|_{L_t^\infty L_x^{\frac{2r}d}}^{1-\frac{\sigma}s},
$$
which is a consequence of interpolation.

Note that the restriction $r\geq d$ guarantees that certain Lebesgue exponents appearing above lie in the range $[1,\infty]$.
In fact, one may relax this restriction a little, but we will not need this here.
\end{proof}

\subsection{Strichartz estimates}

Naturally, everything that we do for the nonlinear Schr\"odinger equation builds on basic properties of the linear propagator
$e^{it\Delta}$.

From the explicit formula
$$ e^{it\Delta} f(x) = \frac{1}{(4\pi i t)^{d/2}} \int_{\R^d} e^{i|x-y|^2/4t} f(y)\, dy,$$
we deduce the standard dispersive inequality
\begin{equation}\label{dispersive}
\| e^{it\Delta} f \|_{L^\infty(\R^d)} \lesssim \frac{1}{|t|^{d/2}} \| f \|_{L^1(\R^d)}
\end{equation}
for all $t\neq 0$.  Interpolating between this and the conservation of mass, gives
\begin{equation}\label{dispersive-p}
\| e^{it\Delta} f \|_{L^p(\R^d)} \lesssim |t|^{\frac dp-\frac d2} \| f \|_{L^{p'}(\R^d)}
\end{equation}
for all $t\neq 0$ and $2\leq p\leq \infty$.  Here $p'$ is the dual of $p$, that is, $\frac 1p+\frac 1{p'}=1$.

Finer bounds on the (frequency localized) linear propagator can be derived using stationary phase:

\begin{lemma}[Kernel estimates]\label{lr:propag est L}
For any $m\geq 0$, the kernel of the linear propagator obeys the following estimates:
\begin{equation}\label{lr:large times}
\bigr| ( P_N e^{it\Delta} )(x,y) \bigl| \lesssim_m \begin{cases}
    |t|^{-d/2} & :\ |x-y| \sim N|t| \\[0.5ex]
    \displaystyle \frac{N^d}{|N^2t|^{m} \langle N|x-y| \rangle^{m} } \quad &:\ \text{otherwise}
    \end{cases}
\end{equation}
for $|t|\geq N^{-2}$ and
\begin{equation}\label{lr:small times}
\bigr| ( P_N e^{it\Delta} )(x,y) \bigl| \lesssim_m N^d \bigl\langle N|x-y| \bigr\rangle^{-m}
\end{equation}
for $|t|\leq  N^{-2}$.
\end{lemma}

We also record the following standard Strichartz estimates:

\begin{lemma}[Strichartz]\label{L:strichartz}  Let $I$ be an interval, let $t_0 \in I$, and let $u_0 \in L^2_x(\R^d)$
and $f \in L^{2(d+2)/(d+4)}_{t,x}(I \times \R^d)$, with $d\geq 3$.  Then, the function $u$ defined by
$$ u(t) := e^{i(t-t_0)\Delta} u_0 - i \int_{t_0}^t e^{i(t-t')\Delta} f(t')\ dt'$$
obeys the estimate
$$
\|u \|_{C^0_t L^2_x} + \| u \|_{L^{\frac{2(d+2)}d}_{t,x}} + \|u \|_{L^2_t L^{\frac{2d}{d-2}}_x}
    \lesssim \| u_0 \|_{L^2_x} + \|f\|_{L^{\frac{2(d+2)}{d+4}}_{t,x}},
$$
where all spacetime norms are over $I\times\R^d$.
\end{lemma}

\begin{proof}
See, for example, \cite{gv:Strichartz, Strichartz}.  For the endpoint see \cite{tao:keel}.
\end{proof}

We will also need three variants of the Strichartz inequality.   First, we observe a weighted Strichartz estimate,
which exploits the spherical symmetry heavily in order to obtain spatial decay. It is very useful in regions of space far from
the origin $x=0$.

\begin{lemma}[Weighted Strichartz]\label{L:wes}  Let $I$ be an interval, let $t_0 \in I$, and let $u_0 \in L^2_x(\R^d)$
and $f \in L^{2(d+2)/(d+4)}_{t,x}(I \times \R^d)$ be spherically symmetric.  Then, the function $u$ defined by
$$ u(t) := e^{i(t-t_0)\Delta} u_0 - i \int_{t_0}^t e^{i(t-t')\Delta} f(t')\ dt'$$
obeys the estimate
$$ \bigl\| |x|^{\frac{2(d-1)}q} u \bigr\|_{L^q_t L^{\frac{2q}{q-4}}_x(I \times \R^d)}
\lesssim \| u_0 \|_{L^2_x(\R^d)} + \|f\|_{L^{\frac{2(d+2)}{d+4}}_{t,x}(I \times \R^d)}$$
for all $4\leq q\leq \infty$.
\end{lemma}

\begin{proof}
For $q=\infty$, this corresponds to the trivial endpoint in Strichartz inequality.  We will only prove the result for the
$q=4$ endpoint, since the remaining cases then follow by interpolation.

As in the usual proof of Strichartz inequality, the method of $T T^*$ together with the Christ--Kiselev lemma
and  Hardy--Littlewood--Sobolev inequality reduce matters to proving that
\begin{equation}\label{radial dispersive}
\bigl\| |x|^{\frac{(d-1)}2} e^{it\Delta} |x|^{\frac{(d-1)}2} g \bigr\|_{L^{\infty}_x(\R^d)}
\lesssim |t|^{-\frac12} \| g \|_{L^1_x(\R^d)}
\end{equation}
for all radial functions $g$.

Let $P_\rad$ denote the projection onto radial functions.  Then
$$
[e^{it\Delta} P_\rad](x,y) = (4\pi i t)^{-\frac d2} e^{i\frac{|x|^2+|y|^2}{4t}}
    \int_{S^{d-1}}  e^{-i\frac{|y| \omega\cdot x}{2t}} \,d\sigma(\omega),
$$
where $d\sigma$ denotes the uniform probability measure on the unit sphere $S^{d-1}$.  This integral can be evaluated
exactly in terms the $J_{\frac{d-2}2}$ Bessel function.  Using this, or simple stationary phase arguments, one sees that
$$
\bigl| [e^{it\Delta} P_\rad](x,y) \bigr| \lesssim  |t|^{-\frac d2} \bigl( \tfrac{|y| |x|}{|t|} \bigr)^{-\frac{d-1}2}
    \lesssim |t|^{-\frac12}  |x|^{-\frac{d-1}2} |y|^{-\frac{d-1}2} .
$$
The radial dispersive estimate \eqref{radial dispersive} now follows easily.
\end{proof}

We will rely crucially on a slightly different type of improvement to the Strichartz inequality in the spherically symmetric
case due to Shao~\cite{Shuanglin}, which improves the spacetime decay of the solution after localizing in frequency:

\begin{lemma}[Shao's Strichartz Estimate, {\cite[Corollary~6.2]{Shuanglin}}]\label{L:Shuanglin}
For $f\in L^2_\rad(\R^d)$ we have
\begin{equation}\label{E:Shuanglin}
\| P_N e^{it\Delta} f \|_{L^q_{t,x}(\R\times\R^d)} \lesssim_q N^{\frac d2-\frac{d+2}q} \| f \|_{L^2_x(\R^d)},
\end{equation}
provided $q>\frac{4d+2}{2d-1}$.
\end{lemma}

The last result is a bilinear estimate, which will be useful for controlling interactions between widely separated frequencies.

\begin{lemma}[Bilinear Strichartz]\label{L:bilinear Shao}
For any spacetime slab $I \times \R^d$, any $t_0 \in I$, and any $M, N > 0$, we have
\begin{align*}
\|( P_{\geq N} u) & (P_{\leq M} v) \|_{L^2_{t,x}(I \times \R^d)}\\
& \lesssim_q N^{-\frac 12} M^{\frac{d-1}2} \Bigl(\|P_{\geq N} u(t_0)\|_{L^2} +
    \| (i\partial_t + \Delta) P_{\geq N} u \|_{L^{\frac{2(d+2)}{d+4}}_{t,x}(I\times\R^d)}\Bigr)\\
&\qquad \qquad \times \Bigl(\|P_{\leq M} v(t_0)\|_{L^2} + \| (i\partial_t + \Delta)P_{\leq M} v\|_{L^{\frac{2(d+2)}{d+4}}_{t,x}(I\times\R^d)}\Bigr),
\end{align*}
for all spherically symmetric functions $u, v$ on $I$.
\end{lemma}

\begin{proof}
See \cite[Lemma 2.5]{Monica:thesis}, which builds on earlier versions in \cite{borg:book, ckstt:gwp}.
\end{proof}

\section{The self-similar solution}\label{ss-sec}

In this section we preclude self-similar solutions.  As mentioned in the Introduction, the key ingredient is additional
regularity.

\begin{theorem}[Regularity in the self-similar case]\label{ss-sob-thm}
Let $d\geq 3$ and let $u$ be a spherically symmetric solution to \eqref{nls} that is almost periodic modulo scaling and that
is self-similar in the sense of Theorem~\ref{comp}.  Then $u(t) \in H^s_x(\R^d)$ for all $t \in (0,\infty)$ and all $0\leq s< 1+\frac 4d$.
\end{theorem}

\begin{corollary}[Absence of self-similar solutions]\label{NO-ss}
For $d\geq 3$ there are no non-zero spherically symmetric solutions to \eqref{nls} that are self-similar in the sense of
Theorem~\ref{comp}.
\end{corollary}

\begin{proof}
By Theorem~\ref{ss-sob-thm}, any such solution would obey $u(t) \in H^1_x(\R^d)$ for all $t\in (0,\infty)$. Then, by the $H^1_x$
global well-posedness theory described after Theorem~\ref{main}, there exists a global solution with initial data $u(t_0)$ at
any time $t_0\in(0,\infty)$; recall that we assume $M(u) < M(Q)$ in the focusing case. On the other hand, self-similar solutions
blow up at time $t=0$. These two facts (combined with the uniqueness statement in Theorem~\ref{local}) yield a contradiction.
\end{proof}

The remainder of this section is devoted to proving Theorem~\ref{ss-sob-thm}.

Let $u$ be as in Theorem~\ref{ss-sob-thm}.  For any $A > 0$, we define
\begin{equation}\label{cadef}
\begin{split}
\mathcal{M}(A) &:= \sup_{T > 0} \| u_{>A T^{-1/2}}(T) \|_{L^2_x(\R^d)} \\
\mathcal{S}(A) &:= \sup_{T > 0} \| u_{>A T^{-1/2}} \|_{L^{2(d+2)/d}_{t,x}([T,2T] \times \R^d)}\\
\mathcal{N}(A) &:= \sup_{T > 0} \| P_{>A T^{-1/2}}F(u) \|_{L^{2(d+2)/(d+4)}_{t,x}([T,2T] \times \R^d)}.
\end{split}
\end{equation}
The notation chosen indicates the quantity being measured, namely, the mass, the symmetric Strichartz norm, and the nonlinearity
in the adjoint Strichartz norm, respectively.  As $u$ is self-similar, $N(t)$ is comparable to $T^{-1/2}$ for $t$ in the
interval $[T,2T]$.  Thus, the Littlewood-Paley projections are adapted to the natural frequency scale on each dyadic time interval.

To prove Theorem~\ref{ss-sob-thm} it suffices to show that for every $0<s<1+\frac 4d$ we have
\begin{equation}\label{c2-targ}
\mathcal{M}(A) \lesssim_{s,u} A^{-s},
\end{equation}
whenever $A$ is sufficiently large depending on $u$ and $s$.  To establish this, we need a variety of estimates linking
$\mathcal{M}$, $\mathcal{S}$, and $\mathcal{N}$.  From mass conservation, Lemma~\ref{spacelemma}, self-similarity, and
H\"older's inequality, we see that
\begin{equation}\label{ctrivb}
\mathcal{M}(A) + \mathcal{S}(A) + \mathcal{N}(A)\lesssim_u 1
\end{equation}
for all $A > 0$. From the Strichartz inequality (Lemma~\ref{L:strichartz}), we also see that
\begin{equation}\label{ci4}
\mathcal{S}(A) \lesssim \mathcal{M}(A) + \mathcal{N}(A)
\end{equation}
for all $A > 0$.  Another application of Strichartz shows
\begin{align}\label{endpoint strich}
\|u\|_{L_t^2 L_x^{\frac{2d}{d-2}}([T,2T]\times\R^d)}\lesssim_u 1.
\end{align}

Next, we obtain a deeper connection between these quantities.

\begin{lemma}[Nonlinear estimate]\label{nle}  Let $\eta>0$ and $0<s<1+\frac 4d$.  For all $A>100$ and $0<\beta\leq 1$, we have
\begin{equation}\label{nle est}
\begin{aligned}
\mathcal{N}(A)
&\lesssim_u \sum_{N\leq \eta A^\beta}\bigl( \tfrac NA\bigr)^s \mathcal{S}(N)
    + \bigl[\mathcal{S}(\eta A^{\frac \beta{2(d-1)}}) + \mathcal{S}(\eta A^\beta)\bigr]^{\frac 4d} \mathcal{S}(\eta A^\beta)\\
&\qquad \qquad + A^{-\frac{2\beta}{d^2}} \bigl[\mathcal{M}(\eta A^\beta) + \mathcal{N}(\eta A^\beta)\bigr].
\end{aligned}
\end{equation}
\end{lemma}

\begin{proof}
Fix $\eta>0$ and $0<s<1+\frac 4d$.  It suffices to bound
$$\| P_{>A T^{-\frac 12}}F(u) \|_{L^{2(d+2)/(d+4)}_{t,x}([T,2T] \times \R^d)}$$
by the right-hand side of \eqref{nle est} for arbitrary $T > 0$ and all $A>100$ and $0<\beta\leq 1$.

To achieve this, we decompose
\begin{equation}\label{1 decomp}
\begin{aligned}
F(u)&=F(u_{\leq \eta A^\beta T^{-\frac 12}}) + O\bigl(|u_{\leq \eta A^\alpha T^{-\frac 12}}|^{\frac 4d} |u_{>\eta A^\beta T^{-\frac 12}}|\bigr)\\
    &\quad + O\bigl(|u_{\eta A^\alpha T^{-\frac 12}< \cdot\leq \eta A^\beta T^{-\frac 12}}|^{\frac 4d} |u_{> \eta A^\beta T^{-\frac 12}}|\bigr)
        + O\bigl(|u_{> \eta A^\beta T^{-\frac 12}}|^{1+\frac 4d}\bigr),
\end{aligned}
\end{equation}
where $\alpha=\tfrac\beta{2(d-1)}$. To estimate the contribution from the last two terms in the expansion above,
we discard the projection to high frequencies and then use H\"older's inequality and \eqref{cadef}:
\begin{align*}
\bigl\||u_{\eta A^\alpha T^{-\frac 12}< \cdot\leq \eta A^\beta T^{-\frac 12}}|^{\frac 4d}
    u_{> \eta A^\beta T^{-\frac 12}} \bigr\|_{L^{\frac{2(d+2)}{d+4}}_{t,x}\!([T,2T] \times \R^d)}
        &\lesssim \mathcal{S}(\eta A^\alpha)^{\frac 4d}\mathcal{S}(\eta A^\beta)\\
\bigl\||u_{> \eta A^\beta T^{-\frac 12}}|^{1+\frac 4d}\bigr\|_{L^{\frac{2(d+2)}{d+4}}_{t,x}\!([T,2T] \times \R^d)}
        &\lesssim \mathcal{S}(\eta A^\beta)^{1+\frac 4d}.
\end{align*}
To estimate the contribution coming from second term on the right-hand side of \eqref{1 decomp},
we discard the projection to high frequencies and then use H\"older's inequality, Lemma~\ref{Bernstein},
Lemma~\ref{L:bilinear Shao}, and \eqref{ci4}:
\begin{align*}
\bigl\|  &  P_{>A T^{-\frac 12}} O\bigl(|u_{\leq \eta A^\alpha T^{-\frac 12}}|^{\frac 4d}
    |u_{> \eta A^\beta T^{-\frac 12}}|\bigr)\bigr\|_{L^{\frac{2(d+2)}{d+4}}_{t,x}\!([T,2T] \times \R^d)}\\
&\lesssim \bigl\|u_{\leq \eta A^\alpha T^{-\frac 12}}u_{>\eta A^\beta T^{-\frac 12}}\bigr\|_{L^2_{t,x}([T,2T] \times \R^d)}^{\frac 8{d^2}}
    \bigl\|u_{> \eta A^\beta T^{-\frac 12}}\bigr\|_{L^{\frac{2(d+2)}d}_{t,x}\!([T,2T] \times \R^d)}^{1-\frac 8{d^2}}\\
&\qquad \qquad \times \bigl\|u_{\leq \eta A^\alpha T^{-\frac 12}}\bigr\|_{L^2_{t,x}([T,2T] \times \R^d)}^{\frac 4d-\frac 8{d^2}}\\
&\lesssim_u \bigl[(\eta A^\beta T^{-\frac 12})^{-\frac 12}(\eta A^\alpha T^{-\frac 12})^{\frac{d-1}2}\bigr]^{\frac 8{d^2}}
    \bigl[\mathcal{M}(\eta A^\beta) + \mathcal{N}(\eta A^\beta)\bigr]^{\frac 8{d^2}}
    \mathcal{S}(\eta A^\beta)^{1-\frac 8{d^2}} T^{\frac 2d-\frac 4{d^2}}\\
&\lesssim_u A^{-\frac{2\beta}{d^2}} \bigl[\mathcal{M}(\eta A^\beta) + \mathcal{N}(\eta A^\beta)\bigr].
\end{align*}
We now turn to the first term on the right-hand side of \eqref{1 decomp}. By Lemma~\ref{Bernstein} and Corollary~\ref{cor:s
deriv} combined with \eqref{ctrivb}, we estimate
\begin{align*}
\| P_{>A T^{-\frac 12}}&F(u_{\leq \eta A^\beta T^{-\frac 12}}) \|_{L^{\frac{2(d+2)}{d+4}}_{t,x}([T,2T] \times \R^d)}\\
&\lesssim (AT^{-\frac 12})^{-s}\bigl\| |\nabla|^s F(u_{\leq \eta A^\beta T^{-\frac 12}}) \bigr\|_{L^{\frac{2(d+2)}{d+4}}_{t,x}([T,2T] \times \R^d)}\\
&\lesssim_u (AT^{-\frac 12})^{-s}\bigl\| |\nabla|^s u_{\leq \eta A^\beta T^{-\frac 12}} \bigr\|_{L^{\frac{2(d+2)}{d}}_{t,x}([T,2T] \times \R^d)}\\
&\lesssim_u \sum_{N\leq \eta A^\beta}\bigl( \tfrac NA\bigr)^s \mathcal{S}(N),
\end{align*}
 which is acceptable.  This finishes the proof of the
lemma.
\end{proof}

We have some decay as $A \to \infty$:

\begin{lemma}[Qualitative decay]\label{qualit lemma}  We have
\begin{equation}\label{clima}
\lim_{A \to \infty} \mathcal{M}(A) = \lim_{A \to \infty} \mathcal{S}(A)=\lim_{A \to \infty} \mathcal{N}(A)=0.
\end{equation}
\end{lemma}

\begin{proof} The vanishing of the first limit follows from Definition~\ref{apdef}, self-similarity, and \eqref{cadef}.
By interpolation, \eqref{cadef}, and \eqref{endpoint strich},
$$
\mathcal{S}(A)
\lesssim \mathcal{M}(A)^{\frac{2}{d+2}} \|u_{\geq A T^{-\frac 12}}\|_{L_t^2 L_x^{\frac{2d}{d-2}}([T,2T]\times\R^d)}^{\frac{d}{d+2}}
\lesssim_u \mathcal{M}(A)^{\frac{2}{d+2}}.
$$
Thus, as the first limit in \eqref{clima} vanishes, we obtain that the second limit vanishes.
The vanishing of the third limit follows from that of the second and Lemma~\ref{nle}.
\end{proof}

We have now gathered enough tools to prove some regularity, albeit in the symmetric Strichartz space. As such, the next result
is the crux of this section.

\begin{proposition}[Quantitative decay estimate]\label{quant prop}
Let $0 < \eta < 1$ and $0<s<1+\frac 4d$.  If $\eta$ is sufficiently small depending on $u$ and $s$,
and $A$ is sufficiently large depending on $u$, $s$, and $\eta$,
\begin{align}\label{recur}
\mathcal{S}(A) \leq \sum_{N\leq \eta A}\bigl( \tfrac NA\bigr)^s \mathcal{S}(N) + A^{-\frac{1}{d^2}}.
\end{align}
In particular,
\begin{align}\label{S decay}
\mathcal{S}(A)\lesssim_u A^{-\frac{1}{d^2}},
\end{align}
for all $A>0$.
\end{proposition}

\begin{proof} Fix $\eta\in (0,1)$ and $0<s<1+\frac 4d$.  To establish \eqref{recur}, it suffices to show
\begin{align}\label{recur-0}
\| u_{>A T^{-1/2}} \|_{L^{\frac{2(d+2)}d}_{t,x}([T,2T] \times \R^d)}
\lesssim_u \sum_{N\leq \eta A}\bigl( \tfrac NA\bigr)^{s+\eps} \mathcal{S}(N) + A^{-\frac{3}{2d^2}}
\end{align}
for all $T > 0$ and some small $\eps>0$, since then \eqref{recur} follows by requiring $\eta$ to be small
and $A$ to be large, both depending upon $u$.

Fix $T>0$.  By writing the Duhamel formula \eqref{old duhamel} beginning at $\frac T2$ and then using Lemma~\ref{L:strichartz},
we obtain
\begin{align*}
\| u_{>A T^{-1/2}} \|_{L^{\frac{2(d+2)}d}_{t,x}([T,2T] \times \R^d)}
&\lesssim \| P_{>A T^{-1/2}} e^{i(t-\frac T2)\Delta} u(\tfrac T2) \|_{L^{\frac{2(d+2)}d}_{t,x}([T,2T] \times \R^d)}\\
&\qquad + \| P_{>A T^{-1/2}} F(u) \|_{L^{\frac{2(d+2)}{d+4}}_{t,x}([\frac T2,2T] \times \R^d)}.
\end{align*}

First, we consider the second term.  By \eqref{cadef}, we have
$$ \| P_{>A T^{-1/2}} F(u) \|_{L^{\frac{2(d+2)}{d+4}}_{t,x}([\frac T2,2T] \times \R^d)} \lesssim \mathcal{N}(A/2).$$
Using Lemma~\ref{nle} (with $\beta=1$ and $s$ replaced by $s+\eps$ for some $0<\eps<1+\frac 4d-s$) combined with Lemma~\ref{qualit lemma}
(choosing $A$ sufficiently large depending on $u$, $s$, and $\eta$), and \eqref{ctrivb}, we derive
$$
\| P_{>A T^{-1/2}} F(u) \|_{L^{\frac{2(d+2)}{d+4}}_{t,x}([\frac T2,2T] \times \R^d)} \lesssim_u \text{RHS\eqref{recur-0}}.
$$
Thus, the second term is acceptable.

We now consider the first term.  It suffices to show
\begin{align}\label{ss decay}
\| P_{>A T^{-1/2}} e^{i(t-\frac T2)\Delta} u(\tfrac T2) \|_{L^{\frac{2(d+2)}d}_{t,x}([T,2T] \times \R^d)}
\lesssim_u A^{-\frac{3}{2d^2}},
\end{align}
which we will deduce by first proving two estimates at a single frequency scale, interpolating between them, and then summing.

From Lemma \ref{L:Shuanglin} and mass conservation, we have
\begin{align}\label{ss Shuanglin}
\| P_{B T^{-1/2}} e^{i(t-\frac T2)\Delta} u(\tfrac T2) \|_{L^q_{t,x}([T,2T] \times \R^d)}
\lesssim_{u,q} (B T^{-1/2})^{\frac d2-\frac{d+2}{q}}
\end{align}
for all $\frac{4d+2}{2d-1} < q \leq \frac{2(d+2)}d$ and $B>0$.  This is our first estimate.

Using the Duhamel formula \eqref{old duhamel}, we write
$$ P_{B T^{-1/2}} e^{i(t-\frac T2)\Delta} u(\tfrac T2) = P_{B T^{-1/2}} e^{i(t-\eps)\Delta} u(\eps)
- i \int_\eps^{\frac T2} P_{B T^{-1/2}} e^{i(t-t')\Delta} F(u(t'))\, dt'$$ for any $\eps > 0$. By self-similarity, the former
term converges strongly to zero in $L_x^2$ as $\eps \to 0$. Convergence to zero in $L_x^{2d/(d-2)}$ then follows from
Lemma~\ref{Bernstein}. Thus, using H\"older's inequality followed by the dispersive estimate \eqref{dispersive-p},
and then \eqref{endpoint strich}, we estimate
\begin{align*}
\| P_{B T^{-1/2}} e^{i(t-\frac T2)\Delta} & u(\tfrac T2)\|_{L^{\frac{2d}{d-2}}_{t,x}([T,2T] \times \R^d)} \\
&\lesssim T^{\frac{d-2}{2d}}\Bigl\|\int_0^{\frac T2} \frac 1{t-t'} \|F(u(t'))\|_{L_x^{\frac{2d}{d+2}}}\, dt'\Bigl\|_{L_t^\infty([T,2T])}\\
&\lesssim T^{-\frac{d+2}{2d}} \| F(u) \|_{L^1_tL_x^{\frac{2d}{d+2}}((0,\frac T2] \times \R^d)} \\
&\lesssim T^{-\frac{d+2}{2d}} \sum_{0<\tau\leq\frac T4} \| F(u) \|_{L^1_tL_x^{\frac{2d}{d+2}}([\tau,2\tau] \times \R^d)} \\
&\lesssim T^{-\frac{d+2}{2d}} \sum_{0<\tau\leq\frac T4} \tau^{1/2} \| u \|_{L^2_tL_x^{\frac{2d}{d-2}}([\tau,2\tau] \times \R^d)}
        \| u \|_{L^\infty_tL^2_x([\tau,2\tau] \times \R^d)}^{\frac 4d} \\
&\lesssim_u T^{-1/d}.
\end{align*}

Interpolating between the estimate just proved and \eqref{ss Shuanglin} with $q=\frac{2d(d+2)(4d-3)}{4d^3-3d^2+12}$, we obtain
$$
\| P_{B T^{-1/2}} e^{i(t-\frac T2)\Delta} u(\tfrac T2) \|_{L^{\frac{2(d+2)}d}_{t,x}([T,2T] \times \R^d)}
\lesssim_u B^{-\frac{3}{2d^2}}.
$$
Summing this over dyadic $B\geq A$ yields \eqref{ss decay} and hence \eqref{recur-0}.

We now justify \eqref{S decay}.  Given an integer $K\geq 4$, we set $\eta=2^{-K}$.
Then, there exists $A_0$ depending on $u$ and $K$, so that \eqref{recur}
holds for $A\geq A_0$.  By \eqref{ctrivb}, we need only bound $\mathcal{S}(A)$ for $A\geq A_0$.

Let $k\geq 0$ and set $A=2^k A_0$ in \eqref{recur}.  Then, writing $N=2^{l}A_0$ and using \eqref{ctrivb},
\begin{align*}
\mathcal{S}(2^k A_0)
&\leq \sum_{l\leq k-K} 2^{-(k-l)s}\mathcal{S}(2^{l}A_0) + (2^kA_0)^{-\beta}\\
&\leq \sum_{l=0}^{k-K} 2^{-(k-l)s}\mathcal{S}(2^{l}A_0) + \frac{2^{-ks}}{1-2^{-s}}\mathcal{S}(0) + 2^{-k\beta}A_0^{-\beta},
\end{align*}
where $\beta:=d^{-2}$.  Setting $s=1$ and applying Lemma~\ref{Gronwall}
with $x_k=\mathcal{S}(2^k A_0)$ and $b_k=O_u(2^{-k\beta})$, we deduce
$$
\mathcal{S}(2^k A_0)\lesssim_u 2^{-k/d^2},
$$
provided $K$ is chosen sufficiently large.  This gives the necessary bound on $\mathcal{S}$.
\end{proof}

\begin{corollary}\label{decay M,S,N}
For any $A>0$ we have
\begin{align*}
\mathcal{M}(A)+\mathcal{S}(A)+\mathcal{N}(A)\lesssim_u A^{-1/d^2}.
\end{align*}
\end{corollary}

\begin{proof}
The bound on $\mathcal{S}$ was proved in the previous proposition.
The bound on $\mathcal{N}$ follows from this, Lemma~\ref{nle} with $\beta=1$, and \eqref{ctrivb}.

We now turn to the bound on $\mathcal{M}$.  By Lemma~\ref{duhamel L},
\begin{align}\label{ss forw}
\|P_{>AT^{-1/2}}u(T)\|_2 \lesssim \sum_{k=0}^\infty \Bigl\| \int_{2^kT}^{2^{k+1}T} e^{i(T-t')\Delta} P_{>AT^{-1/2}}
F(u(t'))\,dt' \Bigr\|_2,
\end{align}
where weak convergence has become strong convergence because of the frequency projection and the fact that $N(t)=t^{-1/2}\to 0$
as $t\to \infty$.  Intuitively, the reason for using \eqref{duhamel} forward in time is that the solution becomes smoother as
$N(t)\to 0$.

Combining \eqref{ss forw} with Lemma~\ref{L:strichartz} and \eqref{cadef}, we get
\begin{align}\label{N-->M}
\mathcal{M}(A)=\sup_{T>0}\|P_{>AT^{-1/2}}u(T)\|_2 \lesssim \sum_{k=0}^\infty \mathcal{N}(2^{k/2}A).
\end{align}
The desired bound on $\mathcal{M}$ now follows from that on $\mathcal{N}$.
\end{proof}

\begin{proof}[Proof of Theorem~\ref{ss-sob-thm}]
Let $0<s<1+\frac 4d$.  Combining Lemma~\ref{nle} (with $\beta= 1-\frac 1{2d^2}$),
\eqref{ci4}, and \eqref{N-->M}, we deduce that if
$$
\mathcal{S}(A)+\mathcal{M}(A)+\mathcal{N}(A) \lesssim_u A^{-\sigma}
$$
for some $0<\sigma<s$, then
$$
\mathcal{S}(A)+\mathcal{M}(A)+\mathcal{N}(A)
\lesssim_u A^{-\sigma}\Bigl(A^{-\frac{s-\sigma}{2d^2}} + A^{-\frac{(d+1)(3d-2)\sigma}{2d^3(d-1)}}
    + A^{-\frac{3 -\sigma}{2d^2}-\frac{d^2-2}{2d^4}}\Bigr).
$$
More precisely, Lemma~\ref{nle} provides the bound on $\mathcal{N}(A)$, then \eqref{N-->M} gives the bound on $\mathcal{M}(A)$
and then finally \eqref{ci4} gives the bound on $\mathcal{S}(A)$.

Iterating this statement shows that $u(t)\in H^{s}_x(\R^d)$ for all $0<s<1+\frac 4d$.
Note that Corollary~\ref{decay M,S,N} allows us to begin the iteration with $\sigma=d^{-2}$.
\end{proof}

\section{An in/out decomposition}\label{S:in/out}

In this section, we will often write radial functions on $\R^d$ just in terms of the radial variable.  With this convention,
\begin{align*}
f(r)= r^{\frac{2-d}{2}} \int_0^\infty J_{\frac{d-2}2}(k r)  \hat f(k) \,k^{\frac d2}\, dk
\quad\text{and}\quad
\hat f(k)= k^{\frac{2-d}{2}} \int_0^\infty J_{\frac{d-2}2}(k r)  f(r) \,r^{\frac d2}\, dr,
\end{align*}
as can be seen from \cite[Theorem~IV.3.3]{stein:weiss}.  Here $J_\nu$ denotes the Bessel function of order $\nu$.
In particular, $g(k,r):=r^{\frac{2-d}2} J_{\frac{d-2}2}(k r)$ solves the radial Helmholtz equation
\begin{equation}\label{Helmholtz}
- g_{rr} - \tfrac{d-1}{r} g_r = k^2 g,
\end{equation}
which corresponds to the fact that $g(k,r)$ represents a spherical standing wave of frequency $k^2/(2\pi)$.
Incoming and outgoing spherical waves are represented by two further solutions of \eqref{Helmholtz}, namely,
$$
g_-(k,r):= r^{\frac{2-d}2} H_{\frac{d-2}2}^{(2)}(k r)
\quad\text{and}\quad
g_+(k,r):= r^{\frac{2-d}2} H_{\frac{d-2}2}^{(1)}(k r),
$$
respectively.  Note that $g=\frac12 g_+ + \frac12 g_-$.  This leads us to define the projection onto outgoing spherical waves by
\begin{align}\label{P+ defn d}
[P^+ f](r) &= \tfrac12 \int_0^\infty r^{\frac{2-d}{2}} H^{(1)}_{\frac{d-2}2}(k r) \hat f(k)\,k^{\frac d2}\,dk \\
&=\tfrac12 r^{\frac{2-d}{2}} \int_0^\infty
        \biggr[\int_0^\infty H^{(1)}_{\frac{d-2}2}(k r) J_{\frac{d-2}2}^{\vphantom{(}}(k \rho) \,k\,dk\biggr]
        f(\rho) \,\rho^{\frac d2}\, d\rho \notag\\
&=\tfrac12 f(r) + \tfrac{i}{\pi} \int_0^\infty \frac{r^{2-d}\,f(\rho)\,\rho^{d-1}\,d\rho}{r^2-\rho^2}. \notag
\end{align}
In order to derive the last equality we used \cite[\S 6.521.2]{GR} together with analytic continuation.
Similarly, we define the projection onto incoming waves by
\begin{align*}
[P^- f](r) &= \tfrac12 \int_0^\infty r^{\frac{2-d}{2}} H^{(2)}_{\frac{d-2}2}(k r) \hat f(k)\,k^{\frac d2}\,dk \\
&=\tfrac12 f(r) - \tfrac{i}{\pi} \int_0^\infty \frac{r^{2-d}\,f(\rho)\,\rho^{d-1}\,d\rho}{r^2-\rho^2}.
\end{align*}
Note that the kernel of $P^-$ is the complex conjugate of that belonging to $P^+$, as is required by time-reversal symmetry.

We will write $P^\pm_N$ for the product $P^\pm P_N$.

\begin{remark} For $f(\rho)\in L^2(\rho^{d-1}\,d\rho)$,
$$
\int_0^\infty \bigl| f(\rho) \bigr|^2 \rho^{d-1}\,d\rho = \tfrac12 \int \bigl| s^{\frac{d-2}4} f(\sqrt{s}) \bigr|^2 \,ds
$$
and with $t=r^2$,
\begin{align}
\int_0^\infty \frac{r^{2-d}\,f(\rho)\,\rho^{d-1}\,d\rho}{r^2-\rho^2}
&=\tfrac12 t^{-\frac{d-2}4} \int_0^\infty \Bigl(\frac st\Bigr)^{\frac{d-2}4}\,\frac{s^{\frac{d-2}4}f(\sqrt{s})\,ds}{t-s}.
\end{align}
Thus $P^+:L^2(\R^d)\to L^2(\R^d)$ is bounded if and only if the Hilbert transform is bounded in the weighted space
$L^2([0,\infty),t^{-(d-2)/2}\,dt)$.  Thus $P^+$ is unbounded on $L^2(\R^d)$ for $d\geq 4$.
\end{remark}

\begin{lemma}[Kernel estimates]\label{P:kernel est}
For $|x|\gtrsim N^{-1}$ and $t\gtrsim N^{-2}$, the integral kernel
obeys
\begin{equation*}
\bigl| [P^\pm_N e^{\mp it\Delta}](x,y) \bigr| \lesssim \begin{cases}
    \bigl(|x||y|\bigr)^{-\frac{d-1}2}|t|^{-\frac12}  &: \  |y|-|x|\sim  Nt \\[1ex]
    \frac{N^d}{(N|x|)^{\frac{d-1}2}\langle N|y|\rangle^{\frac{d-1}2}} \bigl\langle N^2t + N|x| - N|y| \bigr\rangle^{-m}
            &: \  \text{otherwise}\end{cases}
\end{equation*}
for any $m\geq 0$.  For $|x|\gtrsim N^{-1}$ and $|t|\lesssim N^{-2}$, the integral kernel obeys
\begin{equation*}
\bigl| [P^\pm_N e^{\mp it\Delta}](x,y) \bigr|
    \lesssim  \frac{N^d}{(N|x|)^{\frac{d-1}2}\langle N|y|\rangle^{\frac{d-1}2}} \bigl\langle N|x| - N|y| \bigr\rangle^{-m}
\end{equation*}
for any $m\geq 0$.
\end{lemma}

\begin{proof}
The proof is an exercise in stationary phase.  We will only provide the details for $P_N^+ e^{-it\Delta}$,
the other kernel being its complex conjugate. By \eqref{P+ defn d} we have the following formula for the kernel:
\begin{align}\label{kernel}
[P^+_N e^{-it\Delta}](x,y) =
    \tfrac12 \bigl(|x||y|\bigr)^{-\frac{d-2}2}\int_0^\infty H^{(1)}_{\frac{d-2}2}(k |x|) J_{\frac{d-2}2}^{\vphantom{(}}(k |y|)
    e^{itk^2} \psi\bigl(\tfrac kN\bigr)\,k\,dk
\end{align}
where $\psi$ is the multiplier from the Littlewood--Paley projection.  To proceed, we use the following information
about Bessel/Hankel functions:
\begin{align}\label{Bessel symbol}
J_{\frac{d-2}2}(r) = \frac{a(r) e^{ir}}{\langle r\rangle^{1/2}} + \frac{\bar a(r) e^{-ir}}{\langle r\rangle^{1/2}},
\end{align}
where $a(r)$ obeys the symbol estimates
\begin{equation}\label{symbol type}
\Bigr| \frac{\partial^m a(r)}{\partial r^m} \Bigr| \lesssim \langle r \rangle^{-m}
    \quad \text{for all $m\geq0$.}
\end{equation}
The Hankel function $H^{(1)}_{\frac{d-2}2}(r)$ has a singularity at $r=0$; however, for $r\gtrsim 1$,
\begin{align}\label{Hankel symbol}
H^{(1)}_{\frac{d-2}2}(r) = \frac{b(r) e^{ir}}{r^{1/2}}
\end{align}
for a smooth function $b(r)$ obeying \eqref{symbol type}.  As we assume $|x|\gtrsim N^{-1}$, the singularity
does not enter into our considerations.

Substituting \eqref{Bessel symbol} and \eqref{Hankel symbol} into \eqref{kernel}, we see that a stationary phase point
can only occur in the term containing $\bar a(r)$ and even then only if $|y|-|x|\sim Nt$.  In this case,
stationary phase yields the first estimate.  In all other cases, integration by parts yields the second estimate.

The short-time estimate is also a consequence of \eqref{kernel} and stationary phase techniques.  Since $t$ is so small,
$e^{ik^2t}$ shows no appreciable oscillation and can be incorporated into $\psi(\frac kN)$.
For $\bigl||y|-|x|\bigr|\leq N^{-1}$, the result follows from the naive $L^1$ estimate.
For larger $|x|-|y|$, one integrates by parts $m$ times.
\end{proof}

\begin{lemma}[Properties of $P^\pm$]\label{P:P properties}\leavevmode
\begin{SL}
\item $P^+ + P^- $ acts as the identity on $L^2_\rad(\R^d)$.
\item Fix $N>0$.  For any spherically symmetric function $f\in L_x^2(\R^d)$,
$$
\bigl\|P^\pm P_{\geq N} f \bigr\|_{L^2_x(|x|\geq \frac 1{100} N^{-1})} \lesssim \bigl\| f \bigr\|_{L^2_x(\R^d)}
$$
with an $N$-independent constant.
\end{SL}
\end{lemma}

\begin{proof}  Part (i) is immediate from the definition.

We turn now to part (ii).  We only prove the inequality for $P^+$,
as the result for $P^-$ can be deduced from this. Let $\chi$ be a
non-negative smooth function on $\R^+$ vanishing in a neighborhood
of the origin and obeying $\chi(r)=1$ for $r\geq \frac{1}{100}$.
With this definition and \eqref{P+ defn d},
\begin{align*}
\bigl\|P^\pm P_{\geq N} f \bigr\|_{L^2_x(|x|\geq N^{-1})}^2
&\leq \bigl\|\chi(N|x|)P^\pm P_{\geq N} f \bigr\|_{L^2_x(\R^d)}^2\\
&=\int_0^\infty \Bigl|\int_0^\infty H^{(1)}_{\frac{d-2}2}(kr)\hat f(k) k^{\frac d2}
        \bigl(1-\phi\bigl(\tfrac kN\bigr)\bigr)\, dk\Bigr|^2 \chi(Nr)^2 r\,dr,
\end{align*}
where $\phi$ is the Littlewood-Paley cutoff, as in subsection~\ref{ss:basic}.
Note that by scaling, it suffices to treat the case $N=1$.   Because of the cutoffs, the only non-zero
contribution comes from the region $kr\gtrsim 1$.  This allows us to use the following information about
Hankel functions: for $\rho\gtrsim 1$,
$$
H^{(1)}_{\frac{d-2}2}(\rho) = \bigl(\tfrac{2}{\pi \rho}\bigr)^{\frac12} [1+b(\rho)] e^{i\rho-i(d-1)\frac\pi4}
$$
where $b$ is a symbol of order $-1$, that is,
\begin{equation}\label{symbol type -1}
\Bigr| \frac{\partial^m b(\rho)}{\partial \rho^m} \Bigr| \lesssim \langle \rho \rangle^{-m-1}
    \quad \text{for all $m\geq0$;}
\end{equation}
see for example \cite{GR}.  Note that this is more refined than
formula \eqref{Hankel symbol} used in the previous proof. With these
observations, our goal has been reduced to showing that
\begin{align*}
\int_0^\infty \Bigl|\int_0^\infty e^{ikr} \bigl(1+b(kr)\bigr)
        \bigl(1-\phi(k)\bigr)g(k)\, dk\Bigr|^2 \chi(r)^2\,dr
\lesssim \int_0^\infty |g(k)|^2\,dk
\end{align*}
or, equivalently, that
$$
K(k,k'):=\bigl(1-\phi(k)\bigr)\bigl(1-\phi(k')\bigr) \int_0^\infty e^{i(k-k')r} \bigl(1+b(kr)\bigr)\bigl(1+\bar b(k'r)\bigr) \chi(r)^2 \,dr
$$
is the kernel of a bounded operator on $L^2_k([0,\infty))$.  To this end, we will decompose $K$ as the sum of two kernels,
each of which we can estimate.

First, we consider
$$
K_1(k,k') := \bigl(1-\phi(k)\bigr)\bigl(1-\phi(k')\bigr) \int_0^\infty e^{i(k-k')r} \chi(r)^2\,dr.
$$
Without the prefactors, the integral is the kernel of a bounded Fourier multiplier and so a bounded operator on $L^2_k$.  As $\phi$ is
a bounded function, we may then deduce that $K_1$ is itself the kernel of a bounded operator.

Our second kernel is
$$
K_2(k,k') := \bigl(1-\phi(k)\bigr)\bigl(1-\phi(k')\bigr) \!\int_0^\infty\!\! e^{i(k-k')r}
    \bigl[ b(kr) + \bar b(k'r) + b(kr)\bar b(k'r) \bigr] \chi(r)^2 \,dr,
$$
which we will show is bounded using Schur's test.  Note that the factors in front of the integral ensure that the kernel is zero unless
$k\gtrsim 1$ and $k'\gtrsim 1$.  By integration by parts, we see that
$$
K_2(k,k') \lesssim_m |k-k'|^{-m}
$$
for any $m\geq 1$, which offers ample control away from the
diagonal.  To obtain a good estimate near the diagonal, we need to
break the integral into two pieces.  We do this by writing
$1=\chi(r/R)+(1-\chi(r/R))$, with $R\gg1$.  Integrating by parts
once when $r$ is large and not at all when $r$ is small, leads to
\begin{align*}
K_2(k,k')
&\lesssim \frac{1}{|k-k'|} \int \bigl[ \tfrac1{kr^2} + \tfrac1{k'r^2} + \tfrac1{kk'r^3} \bigr]\chi\bigl(\tfrac rR\bigr)
        + \bigl[ \tfrac1{kr} + \tfrac1{k'r} + \tfrac1{kk'r^2} \bigr]\tfrac1R \chi'\bigl(\tfrac rR\bigr)\,dr \\
&\quad + \int \bigl[ \tfrac1{kr} + \tfrac1{k'r} + \tfrac1{kk'r^2} \bigr] \chi(r)^2 \bigl(1-\chi\bigl(\tfrac rR\bigr)\bigr)\,dr\\
&\lesssim \frac{1}{R|k-k'|} + \log(R).
\end{align*}
Choosing $R=|k-k'|^{-1}$ provides sufficient control near the diagonal to complete the application of Schur's test.
\end{proof}

\section{Additional regularity}\label{glob-sob-sec}

This section is devoted to a proof of

\begin{theorem}[Regularity in the global case]\label{glob-sob-thm}
Let $d\geq 3$ and let $u$ be a global spherically symmetric solution to \eqref{nls} that is almost periodic modulo scaling.
Suppose also that $N(t)\lesssim 1$ for all $t\in\R$.  Then $u \in L^\infty_t H^s_x(\R \times \R^d)$ for all $0\leq s< 1+\frac 4d$.
\end{theorem}

The argument mimics that in \cite{KTV}, though the non-polynomial nature of the nonlinearity introduces several technical complications.
That $u(t)$ is moderately smooth will follow from a careful study of the Duhamel formulae \eqref{duhamel}.
Near $t$, we use the fact that there is little mass at high frequencies, as is implied by the definition of almost periodicity and the
boundedness of the frequency scale function $N(t)$.  Far from $t$, we use the spherical symmetry of the solution.  As
this symmetry is only valuable at large radii, we are only able to exploit it by using
the in/out decomposition described in Section~\ref{S:in/out}.

Let us now begin the proof.  For the remainder of the section, $u$ will denote a solution to \eqref{nls}
that obeys the hypotheses of Theorem~\ref{glob-sob-thm}.

We first record some basic local estimates.  From mass conservation we have
\begin{equation}\label{masst}
\| u \|_{L^\infty_t L^2_x(\R \times \R^d)} \lesssim_u 1,
\end{equation}
while from Definition~\ref{apdef} and the fact that $N(t)$ is bounded we have
$$ \lim_{N \to \infty} \| u_{\geq N} \|_{L^\infty_t L^2_x(\R \times \R^d)}  = 0.$$
From Lemma~\ref{spacelemma} and $N(t)\lesssim 1$, we have
\begin{align}\label{gr 44}
\| u \|_{L^{\frac{2(d+2)}d}_{t,x}(J \times \R^d)} \lesssim_u \langle |J|\rangle^{\frac{d}{2(d+2)}}
\end{align}
for all intervals $J\subset \R$.  By H\"older's inequality, this implies
\begin{align}\label{nonlin finite}
\| F(u) \|_{L^{\frac{2(d+2)}{d+4}}_{t,x}(J \times \R^d)} \lesssim_u \langle |J|\rangle^{\frac{d+4}{2(d+2)}}
\end{align}
and then, by the (endpoint) Strichartz inequality (Lemma \ref{L:strichartz}),
\begin{equation}\label{2infty}
\| u \|_{L^2_t L^{\frac{2d}{d-2}}_x(J \times \R^d)} \lesssim_u \langle |J|\rangle^{\frac 12}.
\end{equation}
More precisely, one first treats the case $|J| = O(1)$ using \eqref{gr 44} and then larger intervals by subdivision.
Similarly, from the weighted Strichartz inequality (Lemma~\ref{L:wes}),
\begin{equation}\label{wait}
\bigl\| |x|^{\frac{d-1}2} u_{N_1\leq \cdot\leq N_2} \bigr\|_{L^4_t L^\infty_x(J \times \R^d)} \lesssim_u \langle |J|\rangle^{\frac 14}
\end{equation}
uniformly in $0<N_1\leq N_2<\infty$.

Now, for any dyadic number $N$, define
\begin{equation}\label{cndef}
\mathcal{M}(N) :=\| u_{\geq N} \|_{L^\infty_t L^2_x(\R \times \R^d)}.
\end{equation}
From the discussion above, we see that $\mathcal{M}(N) \lesssim_u 1$ and
\begin{equation}\label{clim}
\lim_{N \to \infty} \mathcal{M}(N) = 0.
\end{equation}

To prove Theorem~\ref{glob-sob-thm}, it suffices to show $\mathcal{M}(N) \lesssim_{u,s} N^{-s}$
for any $0<s<1+\frac 4d$ and all $N$ sufficiently large depending on $u$ and $s$.  As we will explain
momentarily, this will follow from Lemma~\ref{Gronwall} and the following

\begin{proposition}[Regularity]\label{ueta}
Let $u$ be as in Theorem~\ref{glob-sob-thm}, let $0<s<1+\frac 4d$, and let $\eta>0$ be a small number.  Then
$$ \mathcal{M}(N) \leq N^{-s} + \sum_{M\leq \eta N} \bigl(\tfrac MN\bigr)^s\mathcal{M}(M),$$
whenever $N$ is sufficiently large depending on $u$, $s$, and $\eta$.
\end{proposition}

Indeed given $\eps>0$, let $\eta=2^{-K}$ where $K$ is so large that $2\log(K-1)<\eps(K-1)$.
Let $N_0$ be sufficiently large depending on $u$, $s$, and $K$ so that the inequality in Proposition~\ref{ueta}
holds for $N\geq N_0$.  If we write $r=2^{-s}$, $x_k=\mathcal{M}(2^k N_0)$, and
\begin{align*}
b_k = 2^{-ks}N_0^{-s} + \sum_{l\leq -1} 2^{-s(k-l)}\mathcal{M}(2^l N_0)\lesssim_u  2^{-ks} \lesssim_u 2^{-k(s-\eps)} ,
\end{align*}
then \eqref{Gron rec} holds.  Therefore, $\mathcal{M}(N) \lesssim_{u,s} N^{\eps-s}$ by the last sentence in Lemma~\ref{Gronwall}.

The rest of this section is devoted to proving Proposition~\ref{ueta}.  Fix $0<s<1+\frac 4d$ and $\eta>0$.
Our task is to show that
$$
    \| u_{\geq N}(t_0) \|_{L^2_x(\R^d)} \leq N^{-s} + \sum_{M\leq \eta N} \bigl(\tfrac MN\bigr)^s\mathcal{M}(M)
$$
for all times $t_0$ and all $N$ sufficiently large (depending on $u$, $s$, and $\eta$).
By time translation symmetry, we may assume $t_0=0$.  As noted above, one of the
keys to obtaining additional regularity is Lemma~\ref{duhamel L}.  Specifically, we have
\begin{align}
u_{\geq N}(0) &=\bigl( P^+ + P^- \bigr) u_{\geq N}(0) \label{pm rep}\\
       &= \lim_{T\to\infty} i\int_0^T P^+ e^{-it\Delta} P_{\geq N}F(u(t))\,dt
                 -i \lim_{T\to\infty} \int_{-T}^0 P^- e^{-it\Delta} P_{\geq N}F(u(t))\,dt,\notag
\end{align}
where the limit is to be interpreted as a weak limit in $L^2$.  However, this representation is not
useful for $|x|$ small because the kernels of $P^\pm$ have a strong singularity at $x=0$.  To this end,
we introduce the cutoff $\chi_N(x):=\chi(N|x|)$, where $\chi$ is the characteristic function of $[1,\infty)$.
As short times and large times will be treated differently, we rewrite \eqref{pm rep} as
\begin{align}\label{pm rep'}
\chi_N(& x)u_{\geq N}(0,x) \notag\\
&= i\int_0^\delta \chi_N(x)P^+ e^{-it\Delta} P_{\geq N}F(u(t))\,dt
        -i\int_{-\delta}^0 \chi_N(x)P^- e^{-it\Delta} P_{\geq N}F(u(t))\,dt \notag\\
&\quad+\lim_{T\to\infty}\sum_{M\geq N} i\int_\delta^T      \int_{\R^d} \chi_N(x)[P_M^+ e^{-it\Delta}](x,y) [\tilde P_M F(u(t))](y) \,dy\,dt \\
&\quad-\lim_{T\to\infty}\sum_{M\geq N} i\int_{-T}^{-\delta}\int_{\R^d} \chi_N(x)[P_M^- e^{-it\Delta}](x,y) [\tilde P_M F(u(t))](y) \,dy\,dt,\notag
\end{align}
as weak limits in $L_x^2$.  Note that we also used the identity
$$
P_{\geq N} = \sum_{M\geq N} P_M \tilde P_M,
$$
where $\tilde P_M:=P_{M/2}+P_M+P_{2M}$, because of the way we will estimate the large-time integrals.

The analogous representation for treating small $x$ is
\begin{align}\label{no pm rep'}
\bigl(1- & \chi_N(x)\bigr)u_{\geq N}(0,x) \notag\\
&= \lim_{T\to\infty} i\int_0^T \bigl(1- \chi_N(x)\bigr) e^{-it\Delta} P_{\geq N}F(u(t))\,dt\notag\\
&= i\int_0^\delta \bigl(1-\chi_N(x)\bigr)e^{-it\Delta}  P_{\geq N}F(u(t))\,dt \\
    &\quad + \lim_{T\to\infty} \sum_{M\geq N} i\int_\delta^T\int_{\R^d} \bigl(1-\chi_N(x)\bigr)
    [P_M e^{-it\Delta}](x,y) [\tilde P_M F(u(t))](y) \,dy\,dt, \notag
\end{align}
also as weak limits.

To deal with the poor nature of the limits in \eqref{pm rep'} and \eqref{no pm rep'}, we note that
\begin{equation}\label{weakly closed}
f_T \to f \text{ weakly} \quad\Longrightarrow\quad
    \|f\| \leq \limsup_{T\to\infty} \|f_T\|,
\end{equation}
or equivalently, that the unit ball is weakly closed.

Despite the fact that different representations will be used depending on the size of $|x|$, some
estimates can be dealt with in a uniform manner.  The first such example is a bound on
integrals over short times.

\begin{lemma}[Local estimate]\label{local-lemma}
Let $0<s<1+\frac 4d$.  For any sufficiently small $\eta>0$, there exists $\delta = \delta(u,\eta) > 0$ such that
\begin{equation*}
\Bigl\| \int_0^\delta e^{-it\Delta} P_{\geq N} F(u(t))\,dt \Bigr\|_{L^2_x}
     \leq N^{-s} + \tfrac{1}{10}\sum_{M\leq \eta N} \bigl(\tfrac MN\bigr)^s\mathcal{M}(M),
\end{equation*}
provided $N$ is sufficiently large depending on $u$,  $s$, and $\eta$.  An analogous estimate holds for integration over
$[-\delta,0]$ and after pre-multiplication by $\chi_N P^\pm$.
\end{lemma}

\begin{proof}
By Lemma~\ref{L:strichartz}, it suffices to prove
\begin{align}\label{short times}
\mathcal{N}(N):=\| P_{\geq N} F(u) \|_{L^{\frac{2(d+2)}{d+4}}(J \times \R^d)}
\lesssim_u N^{-s-\eps} + \sum_{M\leq \eta N} \bigl(\tfrac MN\bigr)^{s+\eps}\mathcal{M}(M)
\end{align}
for some small $\eps>0$, any interval $J$ of length $|J| \leq \delta$, and all sufficiently large $N$ depending on
$u$, $s$, and~$\eta$, since the claim would follow by requiring $\eta$ small and $N$ large, both depending on $u$.

From \eqref{clim}, there exists $N_0 = N_0(u,\eta)$ such that
\begin{equation}\label{uetan}
 \| u_{\geq N_0} \|_{L^\infty_t L^2_x(\R \times \R^d)} \leq \eta^{100d^2}.
\end{equation}
Let $N>N_1:=\eta^{-1} N_0$.  We decompose
\begin{equation}\label{decomp 2}
\begin{aligned}
F(u) &= F(u_{\leq \eta N}) + O\bigl(|u_{\leq N_0}|^{\frac 4d} |u_{> \eta N}|\bigr)
        + O\bigl(|u_{N_0 \leq \cdot \leq \eta N}|^{\frac 4d} |u_{>\eta N}|\bigr)\\
&\quad  + O\bigl(|u_{>\eta N}|^{1+\frac 4d}\bigr).
\end{aligned}
\end{equation}

Using Lemma~\ref{Bernstein}, Corollary~\ref{cor:s deriv} together with \eqref{gr 44}, and Lemma~\ref{L:strichartz},
we estimate the contribution of the first term on
the right-hand side of \eqref{decomp 2} as follows:
\begin{align*}
\| P_{\geq N} F(u_{\leq \eta N}) \|_{L^{\frac{2(d+2)}{d+4}}(J \times \R^d)}
&\lesssim N^{-s-3\eps} \bigl\| |\nabla|^{s+3\eps} F(u_{\leq \eta N}) \bigr\|_{L^{\frac{2(d+2)}{d+4}}(J \times \R^d)}\\
&\lesssim_u \langle \delta\rangle^{\frac{2}{d+2}} N^{-s-3\eps}\bigl\||\nabla|^{s+3\eps} u_{\leq \eta N}\bigr\|_{L_{t,x}^{\frac{2(d+2)}d}(J \times \R^d)}\\
&\lesssim_u \langle \delta\rangle^{\frac{2}{d+2}} \sum_{M\leq \eta N} \bigl(\tfrac MN\bigr)^{s+3\eps}[\mathcal{M}(M) + \mathcal{N}(M)]\\
&\lesssim_u \eta^\eps \langle \delta\rangle^{\frac{2}{d+2}} \sum_{M\leq \eta N} \bigl(\tfrac MN\bigr)^{s+2\eps}[\mathcal{M}(M) + \mathcal{N}(M)],
\end{align*}
for any $0<\eps<\frac 13 (1+ \frac 4d - s)$.

To estimate the contribution of the second term on the right-hand side of \eqref{decomp 2}, we use
H\"older's inequality, Lemma~\ref{Bernstein}, and \eqref{gr 44}:
\begin{align*}
\bigl\| O\bigl(|u_{\leq N_0}|^{\frac 4d} & |u_{> \eta N}|\bigr) \bigr\|_{L^{\frac{2(d+2)}{d+4}}(J \times \R^d)}\\
&\lesssim \delta^{\frac 12} \|u_{\leq N_0}\|_{L_{t,x}^{\frac{2(d+2)}d}(J \times \R^d)}^{\frac2d}
        \|u_{\leq N_0}\|_{L_{t,x}^\infty(J \times \R^d)}^{\frac2d} \|u_{> \eta N}\|_{L_t^\infty L_x^2(J \times \R^d)}\\
&\lesssim_u \delta^{\frac 12}\langle \delta\rangle^{\frac1{d+2}} N_0 \mathcal{M}(\eta N).
\end{align*}

Finally, to estimate the contribution of the last two terms on the right-hand side of \eqref{decomp 2},
we use H\"older's inequality, interpolation combined with \eqref{2infty} and \eqref{uetan}, and then
Lemma~\ref{L:strichartz} to obtain
\begin{align*}
\bigl\| O\bigl(|& u_{N_0 \leq \cdot \leq\eta N}|^{\frac 4d}  |u_{> \eta N}|\bigr) \bigr\|_{L^{\frac{2(d+2)}{d+4}}(J \times \R^d)}\\
&\lesssim \|u_{N_0 \leq \cdot \leq \eta N}\|_{L_{t,x}^{\frac{2(d+2)}d}(J \times \R^d)}^{\frac4d}
        \|u_{>\eta N}\|_{L_{t,x}^{\frac{2(d+2)}d}(J \times \R^d)}\\
&\lesssim \|u_{N_0 \leq \cdot \leq \eta N}\|_{L_t^\infty L_x^2(J \times \R^d)}^{\frac8{d(d+2)}}
        \|u_{N_0 \leq \cdot \leq \eta N}\|_{L_t^2 L_x^{\frac{2d}{d-2}}(J \times \R^d)}^{\frac4{d+2}}
        [\mathcal{M}(\eta N) + \mathcal{N}(\eta N)]\\
&\lesssim_u \eta^8 \langle \delta\rangle^{\frac{2}{d+2}}[\mathcal{M}(\eta N) + \mathcal{N}(\eta N)]
\end{align*}
and similarly,
\begin{align*}
\bigl\| O\bigl(|u_{> \eta N}|^{1+\frac 4d}\bigr) \bigr\|_{L^{\frac{2(d+2)}{d+4}}(J \times \R^d)}
&\lesssim_u \eta^8 \langle \delta\rangle^{\frac{2}{d+2}}[\mathcal{M}(\eta N) + \mathcal{N}(\eta N)].
\end{align*}

Putting everything together and taking $\eta$ sufficiently small depending on $u$ and $s$,
then $\delta$ sufficiently small depending upon $N_0$ and $\eta$, we derive
\begin{align}\label{almost}
\mathcal{N}(N)
&\leq \sum_{M\leq \eta N} \bigl(\tfrac MN\bigr)^{s+2\eps}[\mathcal{M}(M) + \mathcal{N}(M)]
\end{align}
for all $N>N_1$ and some (very small) $\eps>0$.  The claim \eqref{short times} follows from this and Lemma~\ref{Gronwall}.
More precisely, let $\eta=2^{-K}$ where $K$ is sufficiently large so that $2\log(K-1)<\eps (K-1)$.
If we write $r=2^{-s-2\eps}$, $x_k=\mathcal{N}(2^k N_1)$, and
\begin{align*}
b_k
&= \sum_{l\leq k-K} 2^{-(s+2\eps)(k-l)}\mathcal{M}(2^l N_1) + \sum_{l\leq -1} 2^{-(s+2\eps)(k-l)}\mathcal{N}(2^l N_1)\\
&\lesssim_u \sum_{l\leq k-K} 2^{-(s+2\eps)(k-l)}\mathcal{M}(2^l N_1) + 2^{-(s+2\eps)k},
\end{align*}
then \eqref{almost} implies \eqref{Gron rec}.  With a few elementary manipulations, \eqref{Gron bound}
implies \eqref{short times}.

The last claim follows from Lemma~\ref{P:P properties} after employing $P_{\geq N}=P_{\geq N/2} P_{\geq N}$.
\end{proof}

To estimate the integrals where $|t|\geq\delta$, we break the region
of $(t,y)$ integration into two pieces, namely, where $|y|\gtrsim
M|t|$ and $|y|\ll M|t|$.  The former is the more significant region;
it contains the points where the integral kernels
$P_Me^{-it\Delta}(x,y)$ and $P_M^\pm e^{-it\Delta}(x,y)$ are large
(see Lemmas~\ref{lr:propag est L} and~\ref{P:kernel est}). More
precisely, when $|x|\leq N^{-1}$, we use \eqref{no pm rep'}; in this
case $|y-x|\sim M|t|$ implies $|y|\gtrsim M|t|$ for $|t|\geq \delta
\geq N^{-2}$.  (This last condition can be subsumed under our
hypothesis $N$ sufficiently large depending on $u$ and $\eta$.) When
$|x|\geq N^{-1}$, we use \eqref{pm rep'}; in this case $|y|-|x|\sim
M|t|$ implies $|y|\gtrsim M|t|$.

The next lemma bounds the integrals over the significant region $|y|\gtrsim M|t|$.  Let $\chi_k$ denote
the characteristic function of the set
$$\{(t,y):\, 2^k\delta\leq |t|\leq 2^{k+1}\delta, \ |y|\gtrsim M|t|\}.$$

\begin{lemma}[Main contribution]\label{ar main}
Let $0<s<1+\frac 4d$, let $\eta>0$ be a small number, and let
$\delta$ be as in Lemma~\ref{local-lemma}.  Then
\begin{align*}
 \sum_{M\geq N}\sum_{k=0}^\infty \Bigl\|
\int_{\R}\!\int_{\R^d}
    [P_M e^{-it\Delta}](x,y)\,\chi_k(t,y) \, & [\tilde P_M F(u(t))](y) \,dy\,dt\Bigr\|_{L^2_x}\\
&\leq \tfrac{1}{10}\sum_{L\leq \eta N} \bigl(\tfrac LN\bigr)^s\mathcal{M}(L)
\end{align*}
for all $N$ sufficiently large depending on $u$, $s$, and $\eta$.  An analogous estimate holds with $P_M$
replaced by $\chi_N P_M^+$ or $\chi_N P_M^-$; moreover, the time integrals may be taken over $[-2^{k+1}\delta,-2^k\delta]$.
\end{lemma}

\begin{proof}
We decompose
\begin{align}\label{decomp 3}
F(u)= F(u_{\leq \eta M}) + O\bigl(|u_{>\eta M}|^{1+\frac 4d} \bigr) + O\bigl( |u_{\leq \eta M}|^{\frac 4d} |u_{>\eta M}| \bigr) .
\end{align}

We first consider the contribution coming from the last two terms in the decomposition above.
By the adjoint Strichartz inequality and H\"older's inequality,
\begin{align*}
\Bigl\| \int_{\R}\!  & \int_{\R^d}[P_M e^{-it\Delta}](x,y) \chi_k(t,y)
    \tilde P_M \bigl[ O\bigl(|u_{>\eta M}|^{\frac{d+4}d} \bigr)
    + O\bigl( |u_{\leq \eta M}|^{\frac 4d} |u_{>\eta M}| \bigr)\bigr](y) \,dy\,dt\Bigr\|_{L^2_x}\\
&\lesssim \bigl\|\chi_k \tilde P_M \bigl[ O\bigl(|u_{>\eta M}|^{\frac{d+4}d} \bigr)
    + O\bigl( |u_{\leq \eta M}|^{\frac 4d} |u_{>\eta M}| \bigr)\bigr]\bigr\|_{L_t^1 L_y^2}\\
&\lesssim (M 2^k \delta )^{-\frac{2(d-1)}d} (2^k\delta)^{\frac{d-1}d} \Bigl[ \bigl\||y|^{\frac{2(d-1)}d}\tilde P_M
    O\bigl(|u_{>\eta M}|^{\frac{d+4}d} \bigr)\bigr\|_{L_t^d L_y^2([2^k\delta, 2^{k+1}\delta]\times\R^d)}\\
&\qquad \qquad \qquad \qquad \quad + \bigl\||y|^{\frac{2(d-1)}d}\tilde P_M
    O\bigl( |u_{\leq \eta M}|^{\frac 4d} |u_{>\eta M}| \bigr)\bigr\|_{L_t^d L_y^2([2^k\delta, 2^{k+1}\delta]\times\R^d)}\Bigr].
\end{align*}
As $\tilde P_M$ is a Mihlin multiplier and $|y|^{\frac{4(d-1)}d}$ is an $A_2$ weight, $\tilde P_M$ is bounded on
$L^2(|y|^{\frac{4(d-1)}d}\,dy)$; see \cite[Ch. V]{stein:large}.  Thus, by H\"older's inequality and \eqref{wait},
\begin{align*}
\Bigl\| \int_{\R}\! & \int_{\R^d}  [P_M e^{-it\Delta}](x,y)\chi_k(t,y)
    \tilde P_M \bigl[ O\bigl(|u_{>\eta M}|^{\frac{d+4}d} \bigr)
    + O\bigl( |u_{\leq \eta M}|^{\frac 4d} |u_{>\eta M}| \bigr)\bigr](y) \,dy\,dt\Bigr\|_{L^2_x}\\
&\lesssim (M 2^k \delta )^{-\frac{2(d-1)}d} (2^k\delta)^{\frac{d-1}d} \Bigl[\bigl\||y|^{\frac{2(d-1)}d}
    |u_{>\eta M}|^{\frac{d+4}d}\bigr\|_{L_t^d L_y^2([2^k\delta, 2^{k+1}\delta]\times\R^d)}\\
&\qquad \qquad \qquad \qquad \qquad \qquad + \bigl\||y|^{\frac{2(d-1)}d}
    |u_{\leq \eta M}|^{\frac 4d} |u_{>\eta M}| \bigr\|_{L_t^d L_y^2([2^k\delta, 2^{k+1}\delta]\times\R^d)}\Bigr]\\
&\lesssim (M 2^k \delta )^{-\frac{2(d-1)}d} (2^k\delta)^{\frac{d-1}d} \|u_{>\eta M}\|_{L_t^\infty L_y^2}
    \Bigl[\bigl\||y|^{\frac{d-1}2}u_{>\eta M}\bigr\|_{L_t^4 L_y^\infty([2^k\delta, 2^{k+1}\delta]\times\R^d)}^{\frac 4d}\\
&\qquad \qquad \qquad \qquad \qquad \qquad \qquad +
     \||y|^{\frac{d-1}2}u_{\leq \eta M}\bigr\|_{L_t^4 L_y^\infty([2^k\delta, 2^{k+1}\delta]\times\R^d)}^{\frac 4d}\Bigr]\\
&\lesssim_u (M 2^k \delta )^{-\frac{2(d-1)}d} (2^k\delta)^{\frac{d-1}d} \mathcal{M}(\eta N) \langle 2^k\delta \rangle^{\frac 1d}.
\end{align*}
Summing first in $k\geq 0$ and then in $M\geq N$, we estimate the contribution of the last two terms
on the right-hand side of \eqref{decomp 3} by
$$
\bigl( N^2\delta \bigr)^{-1+\frac 1d}\mathcal{M}(\eta N).
$$

Next we consider the contribution coming from the first term on the right-hand side of \eqref{decomp 3}.
By the adjoint of the weighted Strichartz inequality in Lemma~\ref{L:wes}, H\"older's inequality,
Corollary~\ref{cor:s deriv}, and Lemma~\ref{Bernstein},
\begin{align*}
\Bigl\| \int_{\R}\!\int_{\R^d}
     & [P_M e^{-it\Delta}](x,y)\, \chi_k(t,y) \, [\tilde P_M F(u_{\leq \eta M}(t))](y) \,dy\,dt\Bigr\|_{L^2_x}\\
&\lesssim (M 2^k \delta )^{-\frac{2(d-1)}q}\bigl\|\chi_k \tilde P_M F(u_{\leq \eta M})\bigr\|_{L_t^{\frac{q}{q-1}}L_y^{\frac{2q}{q+4}}}\\
&\lesssim (M 2^k \delta )^{-\frac{2(d-1)}q} (2^k\delta)^{\frac{q-1}q} M^{-s}
    \bigl\||\nabla|^{s} F(u_{\leq \eta M})\bigr\|_{L_t^\infty L_y^{\frac{2q}{q+4}}}\\
&\lesssim (M 2^k \delta )^{-\frac{2(d-1)}q} (2^k\delta)^{\frac{q-1}q} M^{-s}
        \|u_{\leq \eta M}\|_{L_t^\infty L_y^{\frac{2q}d}}^{\frac4d}
        \bigl\||\nabla|^s u_{\leq \eta M}\bigr\|_{L_t^\infty L_y^2}\\
&\lesssim_u (M 2^k \delta )^{-\frac{2(d-1)}q} (2^k\delta)^{\frac{q-1}q}
       (\eta M)^{\frac{2(q-d)}q}\sum_{L\leq \eta M} \bigl( \tfrac LM \bigr)^s \mathcal{M}(L)\\
&\lesssim_u (M^2 2^k \delta )^{-\frac{2d-q-1}q}
       \sum_{L\leq \eta N} \bigl( \tfrac LN \bigr)^s \mathcal{M}(L)
\end{align*}
provided $q\geq \max\{d,4\}$ and $M\geq N$.  In order to deduce the last inequality, we used the fact that for $M\geq N$,
\begin{equation}\label{ha}
\begin{aligned}
\sum_{L\leq \eta M} \bigl(\tfrac LM\bigr)^s\mathcal{M}(L)
&\leq \sum_{L\leq \eta N} \bigl(\tfrac LN\bigr)^s\mathcal{M}(L) + \sum_{\eta N\leq L\leq \eta M} \bigl(\tfrac LM\bigr)^s\mathcal{M}(L)\\
&\lesssim \sum_{L\leq \eta N} \bigl(\tfrac LN\bigr)^s\mathcal{M}(L) + \eta^s \mathcal{M}(\eta N)\\
&\lesssim \sum_{L\leq \eta N} \bigl(\tfrac LN\bigr)^s\mathcal{M}(L).
\end{aligned}
\end{equation}
Therefore, choosing $q=d+1$,
\begin{align*}
\sum_{M\geq N}\sum_{k=0}^\infty \Bigl\|   \int_\R\! \int_{\R^d}
    [P_M e^{-it\Delta}](x,y)\, \chi_k(t,y) & \, [\tilde P_M F(u_{\leq \eta M}(t))](y) \,dy\,dt\Bigr\|_{L^2_x}\\
&\lesssim (N^2 \delta )^{-\frac{d-2}{d+1}} \sum_{L\leq \eta N} \bigl( \tfrac LN \bigr)^s \mathcal{M}(L).
\end{align*}

Putting everything together we obtain
\begin{align*}
\sum_{M\geq N}\sum_{k=0}^\infty \Bigl\|  & \int_\R\! \int_{\R^d}
    [P_M e^{-it\Delta}](x,y)\, \chi_k(t,y) \, [\tilde P_M F(u(t))](y) \,dy\,dt\Bigr\|_{L^2_x}\\
&\lesssim_u \eta^{-s}\Bigl[\bigl( N^2\delta \bigr)^{-1+\frac 2d} + \bigl( N^2\delta \bigr)^{-1+\frac 1d}
    + \bigl( N^2 \delta \bigr)^{-\frac{d-2}{d+1}}\Bigr]\sum_{L\leq \eta N} \bigl(\tfrac LN\bigr)^s\mathcal{M}(L).
\end{align*}
Choosing $N$ sufficiently large depending on $u$, $\delta$, and $s$ (and hence only on $u$, $\eta$, and $s$), we obtain the desired bound.

The last claim follows from the $L_x^2$-boundedness of $\chi_N P^\pm P_{M}$ (cf. Lemma~\ref{P:P properties})
and the time-reversal symmetry of the argument just presented.
\end{proof}

We turn now to the region of $(t,y)$ integration where $|y|\ll M|t|$.  First, we describe the bounds that we will use
for the kernels of the propagators.  For $|x|\leq N^{-1}$, $|y|\ll M|t|$, and $|t|\geq \delta \gg N^{-2}$,
\begin{equation}\label{PM propig}
|P_M e^{-it\Delta}(x,y)| \lesssim \frac{1}{(M^2|t|)^{50d}} \frac{M^d}{\langle M(x-y)\rangle^{50d}};
\end{equation}
this follows from Lemma~\ref{lr:propag est L} since under these constraints, $|y-x|\ll M|t|$.  For $|x|\geq N^{-1}$
and $y$ and $t$ as above,
\begin{equation}\label{PMpm propig}
|P_M^\pm e^{-it\Delta}(x,y)| \lesssim \frac{1}{(M^2|t|)^{50d}}
        \frac{M^d}{\langle Mx\rangle^{\frac{d-1}2}\langle My\rangle^{\frac{d-1}2}\langle M|x|-M|y|\rangle^{50d}};
\end{equation}
by Lemma~\ref{P:kernel est}.  Note that we have used $|y|-|x|\ll M|t|$ and
$$
\langle M^2|t|+M|x|-M|y| \rangle^{-100d} \lesssim (M^2|t|)^{-50d}  \langle M|x|-M|y| \rangle^{-50d}
$$
in order to simplify the bound.

From \eqref{PM propig} and \eqref{PMpm propig} we see that under the hypotheses set out above,
\begin{equation}\label{combined propig}
|P_M e^{-it\Delta}(x,y)| + |P_M^\pm e^{-it\Delta}(x,y)| \lesssim \frac{1}{(M^2|t|)^{50d}} K_M(x,y),
\end{equation}
where
$$
K_M(x,y) := \frac{M^d}{\langle M(x-y)\rangle^{50d}}
        + \frac{M^d}{\langle Mx\rangle^{\frac{d-1}2}\langle My\rangle^{\frac{d-1}2}\langle M|x|-M|y|\rangle^{50d}}.
$$
Note that by Schur's test, this is the kernel of a bounded operator on $L^2_x(\R^d)$.

Let $\tilde\chi_k$ denote the characteristic function of the set
$$\{(t,y):\, 2^k\delta\leq |t|\leq 2^{k+1}\delta, \ |y|\ll M|t|\}.$$

\begin{lemma}[The tail]\label{ar tail}
Let $0<s<1+\frac 4d$, let $\eta>0$ be a small number, and let $\delta$ be as in Lemma~\ref{local-lemma}.  Then
\begin{equation*}
\sum_{M\geq N}\sum_{k=0}^\infty \Bigl\| \int_{\R}\int_{\R^d}
    \frac{K_M(x,y)}{(M^2|t|)^{50d}}\,\tilde\chi_k(t,y) \, |\tilde P_M F(u(t))|(y) \,dy\,dt\Bigr\|_{L^2_x}
\leq \tfrac{1}{10}\!\!\!\sum_{L\leq \eta N}\!\! \bigl(\tfrac LN\bigr)^s\mathcal{M}(L)
\end{equation*}
for all $N$ sufficiently large depending on $u$, $s$, and $\eta$ (in particular, we require $N\gg \delta^{-1/2}$).
\end{lemma}

\begin{proof}
Using H\"older's inequality, the $L^2$-boundedness of the operator with kernel $K_M$, and Lemma~\ref{Bernstein},
\begin{align*}
\Bigl\| \int_{\R}\int_{\R^d} & \frac{K_M(x,y)}{(M^2|t|)^{50d}}\,\tilde\chi_k(t,y) \,
    |\tilde P_M F(u(t))|(y) \,dy\,dt\Bigr\|_{L^2_x}\\
&\lesssim (M^2 2^k\delta )^{-50d} (2^k\delta)^{\frac{d-2}d} M^{\frac{2(d-2)}d}
    \|\tilde P_M F(u)\|_{L_t^{\frac d2} L^{\frac{2d^2}{d^2 +4d -8}}_x([2^k\delta, 2^{k+1}\delta]\times\R^d)}\\
&\lesssim (M^2 2^k\delta )^{-49d}  \|\tilde P_M F(u)\|_{L_t^{\frac d2} L^{\frac{2d^2}{d^2 +4d -8}}_x([2^k\delta, 2^{k+1}\delta]\times\R^d)}.
\end{align*}

We decompose
\begin{align}\label{decomp 4}
F(u)=F(u_{\leq \eta M}) + O\bigl(|u_{\leq \eta M}|^{\frac 4d} |u_{> \eta M}|\bigr) + O\bigl(|u_{> \eta M}|^{1+\frac 4d}\bigr)
\end{align}
Discarding the projection $\tilde P_M$, we use H\"older and \eqref{2infty} to estimate
\begin{align*}
\bigl\|\tilde P_M O\bigl(|u_{\leq \eta M}|^{\frac 4d}
    |u_{> \eta M}|\bigr)&\bigr\|_{L_t^{\frac d2} L^{\frac{2d^2}{d^2 +4d -8}}_x([2^k\delta, 2^{k+1}\delta]\times\R^d)}\\
&\lesssim \|u_{\leq \eta M}\|_{L_t^2 L^{\frac{2d}{d-2}}_x([2^k\delta, 2^{k+1}\delta]\times\R^d)}^{\frac 4d} \|u_{> \eta M}\|_{L_t^\infty L^2_x}\\
&\lesssim_u \langle 2^k\delta \rangle^{\frac 2d}\mathcal{M}(\eta N)\\
\bigl\|\tilde P_M  O\bigl(|u_{>\eta M}|^{1+\frac 4d}\bigr)&\bigr\|_{L_t^{\frac d2} L^{\frac{2d^2}{d^2+4d-8}}_x([2^k\delta, 2^{k+1}\delta]\times\R^d)}\\
&\lesssim \|u_{> \eta M}\|_{L_t^2 L^{\frac{2d}{d-2}}_x([2^k\delta, 2^{k+1}\delta]\times\R^d)}^{\frac 4d} \|u_{> \eta M}\|_{L_t^\infty L^2_x}\\
&\lesssim_u \langle 2^k\delta \rangle^{\frac 2d} \mathcal{M}(\eta N).
\end{align*}
To estimate the contribution coming from the first term on the right-hand side of \eqref{decomp 4},
we use Lemma~\ref{Bernstein}, Corollary~\ref{cor:s deriv} (with $r=\tfrac{d^2}{d-2}$) combined with
H\"older's inequality in the time variable, \eqref{2infty}, and \eqref{ha}, to estimate
\begin{align*}
\|\tilde P_M F(u_{\leq \eta M})&\|_{L_t^{\frac d2} L^{\frac{2d^2}{d^2 +4d -8}}_x([2^k\delta, 2^{k+1}\delta]\times\R^d)}\\
&\lesssim M^{-s} \bigl\||\nabla|^s F(u_{\leq \eta M})\bigr\|_{L_t^{\frac d2} L^{\frac{2d^2}{d^2 +4d -8}}_x([2^k\delta, 2^{k+1}\delta]\times\R^d)}\\
&\lesssim M^{-s} \bigl\||\nabla|^s u_{\leq \eta M}\bigr\|_{L_t^\infty L^2_x}
    \|u_{\leq \eta M}\|_{L_t^2 L^{\frac{2d}{d-2}}_x([2^k\delta, 2^{k+1}\delta]\times\R^d)}^{\frac 4d}\\
&\lesssim_u \langle 2^k\delta \rangle^{\frac 2d}\sum_{L\leq \eta M} \bigl(\tfrac LM\bigr)^s\mathcal{M}(L)\\
&\lesssim_u \langle 2^k\delta \rangle^{\frac 2d}\sum_{L\leq \eta N} \bigl(\tfrac LN\bigr)^s\mathcal{M}(L)
\end{align*}
for any $M\geq N$.

Putting everything together, we deduce
\begin{align*}
\Bigl\| \int_{\R}\int_{\R^d}\frac{K_M(x,y)}{(M^2|t|)^{50d}}\,\tilde\chi_k(t,y) \,
     & [\tilde P_M  F(u(t))](y)  \,dy\,dt\Bigr\|_{L^2_x}\\
&\lesssim_u (M^2 2^k\delta )^{-49d} \langle 2^k\delta \rangle^{\frac 2d} \eta^{-s}\sum_{L\leq \eta N} \bigl(\tfrac LN\bigr)^s\mathcal{M}(L).
\end{align*}
Summing over $k\geq 0$ and $M\geq N$, we obtain
\begin{align*}
\sum_{M\geq N}\sum_{k=0}^\infty \Bigl\| \int_{\R}\int_{\R^d}
    \frac{K_M(x,y)}{(M^2|t|)^{50d}}\,\tilde\chi_k(t,y) \, & [\tilde P_M F(u(t))](y) \,dy\,dt\Bigr\|_{L^2_x}\\
&\lesssim_u (N^2\delta )^{-49d} \eta^{-s}\sum_{L\leq \eta N} \bigl(\tfrac LN\bigr)^s\mathcal{M}(L).
\end{align*}
The claim follows by choosing $N$ sufficiently large depending on $\delta$, $\eta$, and $s$ (and hence only on $u$, $s$, and $\eta$).
\end{proof}

We have now gathered enough information to complete the

\begin{proof}[Proof of Proposition~\ref{ueta}]
Naturally, we may bound $\|u_{\geq N}\|_{L^2}$ by separately bounding the $L^2$ norm on the ball $\{|x|\leq N^{-1}\}$
and on its complement.   On the ball, we use \eqref{no pm rep'}, while outside the ball we use \eqref{pm rep'}.
Invoking \eqref{weakly closed} and the triangle inequality, we reduce the proof to bounding certain integrals.
The integrals over short times were estimated in Lemma~\ref{local-lemma}.  For $|t|\geq \delta$,
we further partition the region of integration into two pieces.  The first piece, where $|y|\gtrsim M|t|$,
was dealt with in Lemma~\ref{ar main}.  To estimate the remaining piece, $|y|\ll M|t|$, one combines
\eqref{combined propig} and Lemma~\ref{ar tail}.
\end{proof}

\section{The double high-to-low frequency cascade}\label{hilo cascade-sec}

In this section, we use the additional regularity provided by Theorem~\ref{glob-sob-thm} to preclude double high-to-low
frequency cascade solutions.  We argue as in \cite{KTV}.

\begin{proposition}[Absence of double cascades]\label{eliminate-ii}  Let $d\geq 3$.  There are no non-zero global
spherically symmetric solutions to \eqref{nls} that are double high-to-low frequency cascades in the sense of
Theorem~\ref{comp}.
\end{proposition}

\begin{proof}
Suppose to the contrary that there is such a solution $u$.  By Theorem~\ref{glob-sob-thm}, $u$ lies in
$C_t^0H^1_x(\R\times\R^d)$. Hence the energy
\begin{equation*}
E(u) = E(u(t)) = \int_{\R^d} \tfrac{1}{2} |\nabla u(t,x)|^2 + \mu\tfrac{d}{2(d+2)}|u(t,x)|^{2(d+2)/d}\, dx
\end{equation*}
is finite and conserved (see e.g. \cite{caz}).  As we have $M(u) < M(Q)$ in the focusing case, the sharp Gagliardo-Nirenberg
inequality (reproduced here as Theorem~\ref{sGN}) gives
\begin{equation}\label{ume}
\| \nabla u(t) \|_{L^2_x(\R^d)}^2 \sim_u E(u) \sim_u 1
\end{equation}
for all $t \in \R$.  We will now reach a contradiction by proving that $\|\nabla u(t)\|_2\to 0$ along any sequence where
$N(t)\to 0$.  The existence of two such time sequences is guaranteed by the fact that $u$ is a double high-to-low frequency
cascade.

Let $\eta > 0$ be arbitrary.  By Definition~\ref{apdef}, we can find $C = C(\eta,u) > 0$ such that
$$ \int_{|\xi| \geq C N(t)} |\hat u(t,\xi)|^2\, d\xi \leq \eta^2$$
for all $t$.  Meanwhile, by Theorem~\ref{glob-sob-thm}, $u\in C_t^0H_x^s(\R\times\R^d)$ for some $s>1$.  Thus,
$$ \int_{|\xi| \geq C N(t)} |\xi|^{2s} |\hat u(t,\xi)|^2\, d\xi \lesssim_{u} 1$$
for all $t$ and some $s>1$.  Thus, by H\"older's inequality,
$$ \int_{|\xi| \geq C N(t)} |\xi|^2 |\hat u(t,\xi)|^2\, d\xi \lesssim_{u} \eta^{2(s-1)/s}.$$
On the other hand, from mass conservation and Plancherel's theorem we have
$$ \int_{|\xi| \leq C N(t)} |\xi|^2 |\hat u(t,\xi)|^2\, d\xi \lesssim_{u} C^2 N(t)^2.$$
Summing these last two bounds and using Plancherel's theorem again, we obtain
$$ \| \nabla u(t) \|_{L^2_x(\R^d)} \lesssim_{u} \eta^{(s-1)/s} + C N(t)$$
for all $t$.  As $\eta>0$ is arbitrary and there exists a sequence of times $t_n\to\infty$ such that $N(t_n)\to 0$ ($u$ is a
double high-to-low frequency cascade), we conclude $\| \nabla u(t_n) \|_2\to 0$.  This contradicts \eqref{ume}.
\end{proof}

\begin{remark}
As mentioned in \cite{KTV}, the argument presented can be used to rule out non-radial single-sided cascade solutions
that lie in $C_t^0 H^s_x$ for some $s>1$.  (By a single-sided cascade we mean a solution with $N(t)$ bounded on a
semi-infinite interval, say $[T,\infty)$, with $\liminf_{t\to\infty}N(t)=0$.)  For such regular solutions $u$, we may
define the total momentum $\int_{\R^d} \Im( \overline{u} \nabla u )$, which is conserved. By a Galilean transformation, we can
set this momentum equal to zero; thus $\int_{\R^d} \xi |\hat u(t,\xi)|^2\ d\xi = 0$.  From this, mass conservation, and the
uniform $H^s_x$ bound for some $s>1$, one can show that $\xi(t) \to 0$ whenever $N(t) \to 0$.  On the other hand, a modification
of the above argument gives
$$ 1 \sim_u \|\nabla u(t)\|_2 \lesssim \eta^{(s-1)/s} + C\bigl(N(t) + |\xi(t)|\bigr),$$
which is absurd.
\end{remark}

\section{Death of a soliton}\label{soliton-sec}
In this section, we use the additional regularity proved in Theorem~\ref{glob-sob-thm} to rule out the third and final enemy,
the soliton-like solution.  Once again, we follow \cite{KTV}; the method is similar to that in \cite{merlekenig}.
Let
\begin{align}\label{general Morawetz}
M_R(t) := 2 \Im \int_{\R^d} \psi(|x|/R) \bar u(t,x) x \cdot \nabla u(x,t) \,dx
\end{align}
where $\psi$ is a smooth function obeying
$$
\psi(r) = \begin{cases}
1 &: r\leq 1 \\
0 &: r\geq 2
\end{cases}
$$
and $R$ denotes a radius to be chosen momentarily.  For solutions $u$ to \eqref{nls} belonging to $C^0_t H^1_x$, $M_R(t)$
is a well-defined function.  Indeed,
\begin{align}\label{Morawetz bound}
|M_R(t)| \lesssim R \|u(t)\|_2 \|\nabla u(t)\|_2 \lesssim_u R.
\end{align}
An oft-repeated calculation (essentially that in the derivation of the Morawetz and viriel identities) gives the following

\begin{lemma} \label{computation}
\begin{align}
\partial_t M_a(t)&= 8 E(u(t)) \notag \\
    &\quad -\int_{\R^d} \Bigl[\tfrac{d^2-1}{R|x|}\psi'\bigl(\tfrac{|x|}R\bigr)+\tfrac{2d+1}{R^2}\psi''\bigl(\tfrac{|x|}R\bigr)
  + \tfrac{|x|}{R^3}\psi'''\bigl(\tfrac{|x|}R\bigr)\Bigr] |u(t,x)|^2 \,dx \label{M2} \\
    &\quad + 4 \int_{\R^d} \Bigl[ \psi\bigl(\tfrac{|x|}R\bigr) - 1  + \tfrac{|x|}{R}\psi'\bigl(\tfrac{|x|}R\bigr)
        \Bigr] |\nabla u(t,x)|^2 \,dx \label{M3} \\
    &\quad +\tfrac{4\mu}{d+2} \int_{\R^d} \Bigl[ d\Bigl(\psi\bigl(\tfrac{|x|}R\bigr) - 1\Bigr)  + \tfrac{|x|}{R}\psi'\bigl(\tfrac{|x|}R\bigr)
        \Bigr] |u(t,x)|^{\frac{2(d+2)}d} \,dx, \label{M4}
\end{align}
where $E(u)$ is the energy of $u$ as defined in \eqref{energy}.
\end{lemma}

\begin{proposition}[Absence of solitons]\label{eliminate-i}  Let $d\geq 3$.  There are no non-zero global spherically
symmetric solutions to \eqref{nls} that are soliton-like in the sense of Theorem~\ref{comp}.
\end{proposition}

\begin{proof}
Assume to the contrary that there is such a solution $u$.  Then, by Theorem~\ref{glob-sob-thm}, $u\in C_t^0H^s_x$ for some $s>1$.
In particular,
\begin{align}\label{M bound}
|M_R(t)|\lesssim_u R.
\end{align}

Recall that in the focusing case, $M(u)<M(Q)$.  As a consequence, the sharp Gagliardo--Nirenberg inequality (reproduced here as
Theorem~\ref{sGN}) implies that the energy is a positive quantity in the focusing case as well as in the defocusing case.
Indeed,
\begin{align*}
E(u)\gtrsim_u \int_{\R^d} |\nabla u(t,x)|^2 \,dx >0.
\end{align*}

We will show that \eqref{M2} through \eqref{M4} constitute only a small fraction of $E(u)$. Combining this fact with
Lemma~\ref{computation}, we conclude $\partial_t M_R(t)\gtrsim E(u)>0$, which contradicts \eqref{M bound}.

We first turn our attention to \eqref{M2}.  This is trivially bounded by
\begin{align}
|\eqref{M2}|\lesssim_u R^{-2}.\label{M2'}
\end{align}

We now study \eqref{M3} and \eqref{M4}.  Let $\eta>0$ be a small number to be chosen later. By Definition~\ref{apdef} and the
fact that $N(t)=1$ for all $t\in \R$, if $R$ is sufficiently large depending on $u$ and $\eta$, then
\begin{align}\label{sol small mass}
\int_{|x|\geq \frac R4} |u(t,x)|^2\, dx\leq \eta
\end{align}
for all $t\in \R$.  Let $\chi$ denote a smooth cutoff to the region $|x|\geq \tfrac R2$, chosen so that $\nabla \chi$ is bounded
by $R^{-1}$ and supported where $|x|\sim R$.  As $u\in C_t^0H^s_x$ for some $s>1$, using interpolation and \eqref{sol small
mass}, we estimate
\begin{align}
|\eqref{M3}| &\lesssim \|\chi\nabla u(t)\|_2^2 \lesssim \|\nabla (\chi u(t))\|_2^2 + \|u(t) \nabla \chi \|_2^2
\lesssim \|\chi u(t)\|_2^{\frac{2(s-1)}s}\|u(t)\|_{H^s_x}^{\frac2s} + \eta  \notag\\
&\lesssim_u \eta^{\frac{s-1}s} + \eta. \label{M3'}
\end{align}

Finally, we are left to consider \eqref{M4}.  Using the same $\chi$ as above together with the Gagliardo--Nirenberg inequality
and \eqref{sol small mass},
\begin{align}
|\eqref{M4}| &\lesssim\|\chi u(t)\|_{\frac{2(d+2)}d}^{\frac{2(d+2)}d}
\lesssim \|\chi u(t)\|_2^{\frac 4d} \|\nabla (\chi u(t))\|_2^2 \lesssim_u \eta^{\frac 2d}.\label{M4'}
\end{align}

Combining \eqref{M2'}, \eqref{M3'}, and \eqref{M4'} and choosing $\eta$ sufficiently small depending on $u$ and $R$ sufficiently
large depending on $u$ and $\eta$, we obtain
$$
|\eqref{M2}| + |\eqref{M3}| + |\eqref{M4}| \leq \tfrac 1{100} E(u).
$$
This completes the proof of the proposition for the reasons explained in the third paragraph.
\end{proof}

\end{document}